\documentclass[11pt,reqno,sumlimits]{amsart}

\usepackage[utf8]{inputenc}
\usepackage{amssymb, amscd, amsmath, epsfig, mathtools}
\usepackage{amsthm}
\usepackage{enumerate}
\usepackage{xcolor}
\usepackage{scalerel}
\usepackage{soul}
\usepackage{tikz-cd}
\usepackage{esint}

\usepackage[margin=1.0in]{geometry}

\newtheorem{theorem}{Theorem}[section]
\newtheorem*{theorem*}{Theorem}
\newtheorem{definition}{Definition}[section]
\newtheorem{corollary}{Corollary}[section]
\newtheorem{lemma}{Lemma}[section]
\newtheorem{proposition}{Proposition}[section]
\theoremstyle{definition}
\newtheorem{remark}{Remark}[section]

\newcommand{\R}{\mathbb R}

\newcommand{\calC}{\mathcal C}
\newcommand{\calL}{\mathcal L}
\newcommand{\dvol}{ d\text{Vol}_{g}}

\begin{document}

\title[On the Zero First Eigenvalue of the Conformal Laplacian]{The Conformal Laplacian and The Kazdan-Warner Problem: Zero First Eigenvalue Case}
\author[J. Xu]{Jie Xu}
\address{
Department of Mathematics and Statistics, Boston University, Boston, MA, U.S.A.}
\email{xujie@bu.edu}
\address{
Institute for Theoretical Sciences, Westlake University, Hangzhou, Zhejiang Province, China}
\email{xujie67@westlake.edu.cn}

\date{}							

\maketitle

\begin{abstract} In this article, we first show that given a smooth function $ S $ either on closed manifolds $ (M, g) $ or compact manifolds $ (\bar{M}, g) $ with non-empty boundary, both for dimensions at least $ 3 $, the condition $ S \equiv 0 $, or $ S $ changes sign and $ \int_{M} S \dvol < 0 $ (with zero mean curvature if the boundary is not empty), is both the necessary and sufficient condition for prescribing scalar curvature problems within conformal class $ [g] $, provided that the first eigenvalue of the conformal Laplacian is zero. We then extend the same necessary and sufficient condition, in terms of prescribing Gauss curvature function and zero geodesic curvature, to compact Riemann surfaces with non-empty boundary, provided that the Euler characteristic is zero. These results are the first full extensions since the results of Kazdan and Warner \cite{KW2} on 2-dimensional torus, and of Escobar and Schoen \cite{ESS} on closed manifolds for dimensions $ 3 $ and $ 4 $. We then give results of prescribing nonzero scalar and mean curvature problems on $ (\bar{M}, g) $, still with zero first eigenvalue and dimensions at least  $ 3 $. Analogously, results of prescribing Gauss and geodesic curvature problems on compact Riemann surfaces with boundary are given for zero Euler characteristic case. Lastly, we show a generalization of the Han-Li conjecture. Technically the key step for manifolds with dimensions at least $ 3 $ is to apply both the local variational methods, local Yamabe-type equations and a new version of the monotone iteration scheme. The key features include the smoothness of the upper solution, the technical difference between constant and non-constant prescribing scalar curvature functions, etc.
\end{abstract}

\section{Introduction}
In this article, we give necessary and sufficient conditions for the prescribed scalar curvature problems within a conformal class of metrics $ [g] $ on both closed manifolds $ (M, g) $ and compact manifolds $ (\bar{M}, g) $ with non-empty smooth boundary $ \partial M $, with dimensions $ n = \dim M \geqslant 3 $ or $ n = \dim \bar{M} \geqslant 3 $, provided that the first eigenvalue of the conformal Laplacian  (possibly with appropriate Robin boundary condition) is zero. We also give necessary and sufficient conditions for prescribing Gauss curvature problems under conformal deformation on compact Riemann surface with non-empty boundary. This problem was solved on closed Riemann surface by Kazdan and Warner \cite{KW2} in 1974 for the analogously zero Euler characteristic case. When the dimension is either $ 3 $ or $ 4 $, this problem was solved by Escobar and Schoen \cite{ESS} in 1986. Based on our best understanding, this article is the first progress in this direction since then. The results here completely solve this problem for closed manifolds, and for compact manifolds with non-empty, smooth, minimal boundary (zero mean curvature after some conformal change), also on compact Riemann surfaces with zero geodesic curvature after some appropriate conformal change. Our methods, which involves local variational method, construction of lower and upper solutions of a nonlinear PDE, and the monotone iteration scheme, are not dimensionally specific for all manifolds with dimensions at least $ 3 $. On $ 2 $-dimensional case, we apply a variational method since the PDE we are dealing with is different.

In addition, we also show the existence of the positive, smooth solution of a local Yamabe equation with Dirichlet boundary condition, for which the coefficient functions of the nonlinear term $ u^{\frac{n + 2}{n - 2}} $ is non-constant. This result is a generalization of the local Yamabe equation with constant coefficient functions on nonlinear terms \cite{XU6}; this local result plays a key role in solving global problems mentioned above, when the dimensions of the manifolds are at least $ 3 $. We also give new results for prescribing non-zero scalar and mean curvature functions under conformal deformation on $ (\bar{M}, g) $ with zero first eigenvalue of the conformal Laplacian. Due to the new version of the monotone iteration scheme, we show an extension of the Han-Li conjecture, which was first mentioned in \cite{HL} and completely proved in \cite{XU5}.

Given a compact manifold whose dimension is at least $ 3 $, possibly with boundary, the Kazdan-Warner problem considers the prescribed scalar and mean curvature functions for a metric $ \tilde{g} $, which is pointwise conformal to the original metric $ g $. Interchangably we call $ \tilde{g} \in [g] $ be a Yamabe metric. Let $ R_{g} $ and $ h_{g} $ be the scalar curvature and mean curvature of a metric $ g $. Let $ a = \frac{4(n - 1)}{n -2} $, $ p = \frac{2n}{n - 2} $ and let $ -\Delta_{g} $ be the positive definite Laplace-Beltrami operator. We define $ \eta_{1} $ to be the first eigenvalue of the conformal Laplacian $ \Box_{g}u : = -a\Delta_{g} u + R_{g} u $ on closed manifolds $ (M, g) $
\begin{equation*}
\Box_{g} u = -a\Delta_{g} \varphi + R_{g} \varphi = \eta_{1} \varphi \; {\rm in} \; M.
\end{equation*}
Here the positive function $ \varphi \in \calC^{\infty}(M) $ is the associated first eigenfunction. Similarly, we define the first eigenvalue $ \eta_{1}' $ of the conformal Laplacian $ \Box_{g} $ with Robin boundary condition on compact manifolds with non-empty boundary $ \partial M $
\begin{equation*}
\Box_{g} u = -a\Delta_{g} \varphi' + R_{g} \varphi' = \eta_{1}' \varphi' \; {\rm in} \; M, B_{g} u : = \frac{\partial \varphi'}{\partial \nu} + \frac{2}{p - 2} h_{g} \varphi' = 0 \; {\rm on} \; \partial M
\end{equation*}
with the associated first eigenfunction $ \varphi' $.

\medskip

The first two main results for this article are listed as follows, which are proven in Theorem \ref{zero:thm1} and Theorem \ref{zero:thm2}, respectively. We would like to point out here that almost the same methods are applied to both the closed manifolds case and compact manifolds with boundary case; these methods have no boundary issues at all.
\begin{theorem}\label{intro:thm1}
Let $ (M, g) $ be a closed manifold, $ n = \dim M \geqslant 3 $. Let $ S \in \calC^{\infty}(M) $ be a given function. Assume that $ \eta_{1} = 0 $. If the function $ S $ satisfies
\begin{equation*}
\begin{split}
& S \equiv 0; \\
& \text{or $ S $ changes sign and} \; \int_{M} S \dvol < 0,
\end{split}
\end{equation*}
then $ S $ can be realized as a prescribed scalar curvature function of some pointwise conformal metric $ \tilde{g} \in [g] $.
\end{theorem}
\begin{theorem}\label{intro:thm2}
Let $ (\bar{M}, g) $ be a compact manifold with non-empty smooth boundary $ \partial M $, $ n = \dim \bar{M} \geqslant 3 $. Let $ S \in \calC^{\infty}(\bar{M}) $ be a given function. Assume that $ \eta_{1}' = 0 $. If the function $ S $ satisfies
\begin{equation*}
\begin{split}
& S \equiv 0; \\
& \text{or $ S $ changes sign and} \; \int_{M} S \dvol < 0,
\end{split}
\end{equation*}
then $ S $ can be realized as a prescribed scalar curvature function of some pointwise conformal metric $ \tilde{g} \in [g] $ with minimal boundary ($ h_{\tilde{g}} \equiv 0 $).
\end{theorem}
When the manifold is a compact Riemann surface with non-empty smooth boundary, we denote $ K_{g} $ and $ \sigma_{g} $ be the Gauss and geodesic curvature, respectively. The next result, proven in Theorem \ref{de2:thm1}, is the $ 2 $-dimensional analogy of Theorem \ref{intro:thm2}.
\begin{theorem}\label{intro:thm3}
Let $ (\bar{M}, g) $ be a compact Riemann surface with non-empty smooth boundary $ \partial M $. Let $ K \in \calC^{\infty}(\bar{M}) $ be a given function. Assume that $ \chi(\bar{M}) = 0 $. If the function $ K $ satisfies
\begin{equation*}
\begin{split}
& K \equiv 0; \\
& \text{or $ K $ changes sign and} \; \int_{M} K \dvol < 0,
\end{split}
\end{equation*}
then $ K $ can be realized as a prescribed Gauss curvature function of some pointwise conformal metric $ \tilde{g} \in [g] $ with $ \sigma_{\tilde{g}} \equiv 0 $.
\end{theorem}
As introduced in the Yamabe problem \cite{PL}, the Kazdan-Warner problem for prescribing functions $ S, H $ on compact manifolds with dimensions at least $ 3 $ is reduced to the existence of the positive, smooth solution of the following PDEs, where the first one below is for closed manifolds, and the second one below is for compact manifolds with non-empty smooth boundary:
\begin{equation}\label{intro:eqn1}
-a\Delta_{g} u + R_{g} u = S u^{p-1} \; {\rm in} \; M.
\end{equation}
\begin{equation}\label{intro:eqn2}
-a\Delta_{g} u + R_{g} u = S u^{p-1} \; {\rm in} \; M, \frac{\partial u}{\partial \nu} + \frac{2}{p-2} h_{g} u = \frac{2}{p-2} H u^{\frac{p}{2}}.
\end{equation}
We point out that the result of Theorem \ref{intro:thm2} is a special case of (\ref{intro:eqn2}) for $ H \equiv 0 $. One key technique to solve these PDEs is a new version of the monotone iteration scheme for nonlinear elliptic PDEs, which was given in Theorem \ref{pre:thm4}. The monotone iteration scheme requires the constructions of upper and lower solutions. As shown in \cite{XU4, XU5, XU6, XU7, XU3}, we use the solution of some local Yamabe equation to construct the global lower solution, and applying gluing technique between the local solution and some candidate of upper solution to construct the global upper solution. We would like to point out two key features: (a) the upper solution we construct is a smooth function, not just piecewise smooth; (b) the iterative methods are different, depending on whether the prescribed scalar curvature is a globally constant function or not. This procedure is first developed for solving the Yamabe problem, the Escobar problem, and prescribed scalar curvature problems for positive first eigenvalue, or equivalently, positive Yamabe invariant case. The iterative method is inspired by previous work on nonlinear elliptic PDEs and Yamabe-type problems on Euclidean spaces in \cite{XU2, XU}.

On compact Riemann surface with non-empty smooth boundary, the problem of prescribing functions $ K, \sigma \in \calC^{\infty}(\bar{M}) $ is reduced to the existence of some smooth solution $ u $ of the following PDE
\begin{equation}\label{intro:eqn3}
-\Delta_{g} u + K_{g} = K e^{2u} \; {\rm in} \; M, \frac{\partial u}{\partial \nu} + \sigma_{g} u = \sigma e^{u} \; {\rm on} \; \partial M
\end{equation}
When $ \chi(\bar{M}) = 0 $, we apply a variational method due to Kazdan and Warner \cite[Thm.~5.3]{KW2}. The reason to use a different method, instead of the monotone iteration scheme is partially due to the lack of a local solution. Note that $ \chi(\bar{M}) = 0 $ is a $ 2 $-dimensional analogy of the classification $ \eta_{1} = 0 $ or $ \eta_{1}' = 0 $ in higher dimensional case.
\medskip

We also give results for prescribing scalar and non-zero mean curvature functions on $ (\bar{M}, g) $, provided that $ \eta_{1}' = 0 $. The next main result, which was proven in Theorem \ref{zerog:thm1}, Corollary \ref{zerog:cor1} and Corollary \ref{zerog:cor2}, is given below:
\begin{theorem}\label{intro:thm4}
Let $ (\bar{M}, g) $ be a compact manifold with non-empty smooth boundary $ \partial M $, $ n = \dim \bar{M} \geqslant 3 $. Let $ S, H \in \calC^{\infty}(\bar{M}) $ be given nonzero functions. Assume that $ \eta_{1}' = 0 $. If the function $ S $ satisfies
\begin{equation*}
\text{$ S $ changes sign and} \; \int_{M} S \dvol < 0,
\end{equation*}
then there exists a pointwise conformal metric $ \tilde{g} \in [g] $ that has scalar curvature $ R_{\tilde{g}} = S $ and $ h_{\tilde{g}} = cH $ for some small enough positive constant $ c $.
\end{theorem}
Beginning with Aubin \cite{Aubin}, Kazdan and Warner \cite{KW2, KW3, KW}, and Nirenberg \cite{AC}, many people have made great progress in the so-called Kazdan-Warner problem. On closed manifolds, other contribution includes the complete solution of the Yamabe problem due to Yamabe, Trudinger, Aubin and Schoen \cite{PL}, and the prescribed scalar curvature problems on $ \mathbb{S}^{n} $ by Y. Y. Li \cite{YYL}, Schoen and Yau \cite{SY}, Bourguignon and Ezin \cite{BE}, Malchiodi and Mayer \cite{MaMa}, etc. We gave a comprehensive solution of the prescribed scalar curvature problem with positive first eigenvalue in \cite{XU6}.  On compact manifolds with non-empty boundary, Escobar \cite{ESC, ESC2, Escobar2, ESS} introduced the boundary Yamabe problem and discussed many variations in prescribing different types of scalar and mean curvature functions. Other contribution includes the work of Marques, Brendle, X.Z. Chen, etc., see \cite{BM}. For constant scalar and mean curvatures, we gave complete solution of the Escobar problem in \cite{XU4}, and of the Han-Li conjecture in \cite{XU5}. For positive first eigenvalue with Robin condition, we gave a comprehensive discussion on prescribed scalar curvature problem with minimal boundary in \cite{XU6}. For Dirichlet boundary conditions, we gave a generalization of ``Trichotomy Theorem" in \cite{XU7}. 
\medskip

Due to the new monotone iteration scheme in Theorem \ref{pre:thm4}, we can give a generalization of the Han-Li conjecture for positive and negative first eigenvalue cases:
\begin{theorem}\label{intro:thm5}
Let $ (\bar{M}, g) $ be a compact manifold with smooth boundary, $ \dim \bar{M} \geqslant 3 $. Let $ \eta_{1}' $ be the first eigenvalue of the boundary eigenvalue problem $ \Box_{g} = \eta_{1}' u $ in $ M $, $ B_{g} u = 0 $ on $ \partial M $. Then: \\
\begin{enumerate}[(i).]
\item If $ \eta_{1}' < 0 $, (\ref{intro:eqn2}) with constant functions $ S = \lambda \in \R $, $ H = \zeta \in \R $ admits a real, positive solution $ u \in \calC^{\infty}(\bar{M}) $ with some $ \lambda < 0 $ and $ \zeta < 0 $;
\item If $ \eta_{1}' > 0 $, (\ref{intro:eqn2}) with constant functions $ S = \lambda \in \R $, $ H = \zeta \in \R $ admits a real, positive solution $ u \in \calC^{\infty}(\bar{M}) $ with some $ \lambda > 0 $ and $ \zeta < 0 $.
\end{enumerate}
\end{theorem} 
The original Han-Li conjecture \cite{HL} was stated for positive constant mean curvatures. The new monotone iteration scheme released the restriction of the positivity of the prescribed function $ H $ on $ \partial M $.
\medskip

This article is organized as follows:

In \S2, we give essential definition and tools for later sections, like Sobolev spaces, $ W^{s, q} $-type elliptic regularities, etc. We also assume the general backgrounds of elliptic PDE theory and Sobolev embeddings. We then introduce two versions of monotone iteration schemes, one for closed manifolds in Theorem \ref{pre:thm3} and the other for compact manifolds with boundary in Theorem \ref{pre:thm4}.

In \S3, we introduce the theory of the local Yamabe-type equation $ -a\Delta_{g} u + R_{g} u = f u^{p-1} \; {\rm in} \; \Omega, u = 0 \; {\rm on} \; \partial \Omega $ on a small enough Riemannian domain $ (\Omega, g) $. In Proposition \ref{local:prop2}, we show the existence of a positive, smooth solution of the local Yamabe-type equation for non-constant positive function $ f $, which is a generalization of constant function $ f $ given in \cite[Lemma.~3.2]{XU6}. This result holds when $ \dim \Omega \geqslant 3 $ and $ g $ is not locally conformally flat, and thus is not dimensionally specific. When $ g $ is locally conformally flat, we introduce the local solution in Proposition \ref{local:prop3}.

In \S4, we give the necessary and sufficient condition for prescribing scalar curvature function $ S $ on both closed manifolds and compact manifolds with non-empty minimal boundary, with dimensions at least $ 3 $, provided that $ \eta_{1}' = 0 $. The condition, $ S \equiv 0 $ or $ S $ changes sign and $ \int_{M} S \dvol < 0 $, is exactly the same for 2-dimensional torus, and dimensions $ 3 $ and $ 4 $. We construct lower solutions in Proposition \ref{zero:prop3} and upper solutions in Proposition \ref{zero:prop4}. The two major results are given in Theorem \ref{zero:thm1} for closed manifolds $ (M, g) $, and in Theorem \ref{zero:thm2} for compact manifolds $ (\bar{M}, g) $ with non-empty, smooth, minimal boundary.

In \S5, we generalize the prescribing scalar curvature function $ S \in \calC^{\infty}(\bar{M}) $ and mean curvature function $ H \in \calC^{\infty}(\partial M) $ problems discussed in \S4, in which we assume that both $ S $ and $ H $ are non-zero functions. This section begins with an inequality between the average of the scalar curvature and the average of the mean curvature within the conformal class. Using the monotone iteration schemes, we show that if $ S $ changes sign and $ \int_{M} S \dvol < 0 $, then there exists some Yamabe metric $ \tilde{g} \in [g] $ that has scalar curvature $ S $ and mean curvature $ cH $ for some small enough positive constant $ c $. In particular, Theorem \ref{zerog:thm1} associates with the case $ \int_{\partial M} H dS_{g} > 0 $; Corollary \ref{zerog:cor1} is for the case $ \int_{\partial M} H dS_{g} = 0 $; and Corollary \ref{zerog:cor2} is for the case $ \int_{\partial M} H dS_{g} < 0 $. We conjectured that $ S $ changes sign and $ \int_{M} S \dvol < 0 $ is also a necessary condition when $ \eta_{1}' = 0 $, but we can only give partial reasoning.

In \S6, we apply a global variation method to consider necessary and sufficient condition for prescribing Gauss curvature function $ K $ on compact Riemann surface with non-empty boundary for some Yamabe metric $ \tilde{g} \in [g] $ with $ \sigma_{\tilde{g}} = 0 $. The main result in Theorem \ref{de2:thm1} is an extension of the result on $ 2 $-dimensional torus \cite[Thm.~5.3]{KW2}.

In \S7, we give a generalization of the Han-Li conjecture for the following two cases:
\begin{enumerate}[(i).]
\item If $ \eta_{1}' < 0 $, then (\ref{intro:eqn2}) admits a positive, smooth solution with some $ S = \lambda < 0 $ and $ H = \zeta < 0 $; 
\item If $ \eta_{1}' > 0 $, then (\ref{intro:eqn2}) admits a positive, smooth solution with some $ S = \lambda > 0 $ and $ H = \zeta < 0 $.
\end{enumerate}
We show the Case (i) in Theorem \ref{HL:thm2} and the Case (ii) in Theorem \ref{HL:thm3}. Overall, the key step is the new version of the monotone iteration scheme in Theorem \ref{pre:thm4}, which is widely used in proving all major results. In addition, the systematic procedure from local variational method, local Yamabe-type equation, gluing method, and iteration scheme we developed in previous work is also crucial. 

\section{The Preliminaries and The Monotone Iteration Schemes}
In this section, we introduce the Sobolev spaces both on manifolds and on local domains. We then introduce the $ \calL^{q} $-type elliptic regularities, both on closed manifolds and on compact manifolds with non-empty boundary. Next we introduce essential tools for later sections, especially variations of the monotone iteration schemes. The new variation of the monotone iteration scheme is quite useful in dealing with prescribed non-constant scalar and mean curvature problems. We assume the background of standard elliptic theory, the strong and weak maximal principles for second order elliptic operators, Sobolev embeddings for $ W^{s, q} $ -type and $ \calC^{0, \alpha} $-type, the trace theorem, etc.

We begin with a few set-up. Let $ n $ be the dimension of the compact manifold, with or without boundary. Let $ \Omega $ be a connected, bounded, open subset of $ \R^{n} $ with smooth boundary $ \partial \Omega $ equipped with some Riemannian metric $ g $ that can be extended smoothly to $ \bar{\Omega} $. We call $ (\Omega, g) $ a Riemannian domain. Furthermore, let $ (\bar{\Omega}, g) $ be a compact manifold with smooth boundary extended from $ (\Omega, g) $. Throughout this article, we denote $ (M, g) $ to be a closed manifold with $ \dim M \geqslant 3 $, and $ (\bar{M}, g) $ to be a general compact manifold with interior $ M $ and smooth boundary $ \partial M $; we denote the space of smooth functions with compact support by $ \calC_{c}^{\infty} $, smooth functions by $ \calC^{\infty} $, and continuous functions by $ \calC^{0} $.

Foremost, we define Sobolev spaces on $ (M, g) $, $ (\bar{M}, g) $ and $ (\Omega, g) $, both in global expressions and local coordinates.
\begin{definition}\label{pre:def1} Let $ (\Omega, g) $ be a Riemannian domain. Let $ (M, g) $ be a closed  Riemannian $n$-manifold with volume density $\dvol$. Let $u$ be a real valued function. Let $ \langle v,w \rangle_g$ and $ |v|_g = \langle v,v \rangle_g^{1/2} $ denote the inner product and norm  with respect to $g$. 

(i) 
For $1 \leqslant q < \infty $,
\begin{align*}
\mathcal{L}^{q}(\Omega)\ &{\rm is\ the\ completion\ of}\  \left\{ u \in \calC_c^{\infty}(\Omega) : \Vert u\Vert_{\calL^{q}(\Omega)}^q :=\int_{\Omega} \lvert u \rvert^{q} dx < \infty \right\},\\
\mathcal{L}^{q}(\Omega, g)\ &{\rm is\ the\ completion\ of}\ \left\{ u \in \calC_c^{\infty}(\Omega) : \Vert u\Vert_{\calL^{q}(\Omega, g)}^q :=\int_{\Omega} \left\lvert u \right\rvert^{q} d\text{Vol}_{g} < \infty \right\}, \\
\mathcal{L}^{q}(M, g)\ &{\rm is\ the\ completion\ of}\ \left\{ u \in \calC^{\infty}(M) : \Vert u\Vert_{\calL^{q}(M, g)}^q :=\int_{M} \left\lvert u \right\rvert^{q} d\text{Vol}_{g} < \infty \right\}.
\end{align*}

(ii) For $ q = \infty $,
\begin{equation*}
\lVert f \rVert_{\calL^{\infty}(M)} = \inf \lbrace C \geqslant 0 : \lvert f(x) \rvert \leqslant C \; \text{for almost all $ x \in M $} \rbrace.
\end{equation*}

(iii) For $\nabla u$  the Levi-Civita connection of $g$, 
and for $ u \in \calC^{\infty}(\Omega) $ or $ u \in \calC^{\infty}(M) $,
\begin{equation*}
\lvert \nabla^{k} u \rvert_g^{2} := (\nabla^{\alpha_{1}} \dotso \nabla^{\alpha_{k}}u)( \nabla_{\alpha_{1}} \dotso \nabla_{\alpha_{k}} u).
\end{equation*}
\noindent In particular, $ \lvert \nabla^{0} u \rvert^{2}_g = \lvert u \rvert^{2} $ and $ \lvert \nabla^{1} u \rvert^{2}_g = \lvert \nabla u \rvert_{g}^{2}.$\\

(iv) For $ s \in \mathbb{N}, 1 \leqslant p < \infty $,
\begin{align*}
W^{s, q}(\Omega) &= \left\{ u \in \mathcal{L}^{q}(\Omega) : \lVert u \rVert_{W^{s,q}(\Omega)}^{q} : = \int_{\Omega} \sum_{j=0}^{s} \left\lvert D^{j}u \right\rvert^{q} dx < \infty \right\}, \\
W^{s, q}(\Omega, g) &= \left\{ u \in \mathcal{L}^{q}(\Omega, g) : \lVert u \rVert_{W^{s, q}(\Omega, g)}^{q} = \sum_{j=0}^{s} \int_{\Omega} \left\lvert \nabla^{j} u \right\rvert^{q}_g \dvol < \infty \right\}, \\
W^{s, q}(M, g) &= \left\{ u \in \mathcal{L}^{q}(M, g) : \lVert u \rVert_{W^{s, q}(M, g)}^{q} = \sum_{j=0}^{s} \int_{M} \left\lvert \nabla^{j} u \right\rvert^{q}_g \dvol < \infty \right\}.
\end{align*}
\noindent Here $ \lvert D^{j}u \rvert^{q} := \sum_{\lvert \alpha \rvert = j} \lvert \partial^{\alpha} u \rvert^{q} $ in the weak sense. Similarly, $ W_{0}^{s, q}(\Omega) $ is the completion of $ \calC_{c}^{\infty}(\Omega) $ with respect to the $ W^{s, q} $-norm. 

In particular, $ H^{s}(\Omega) : = W^{s, 2}(\Omega) $ and $ H^{s}(\Omega, g) : = W^{s, 2}(\Omega, g) $, $ H^{s}(M, g) : = W^{s, 2}(M, g) $ are the usual Sobolev spaces. We similarly define $H_{0}^{s}(\Omega), H_{0}^{s}(\Omega,g)$.

(v) We define the $ W^{s, q} $-type Sobolev space on $ (\bar{M}', g) $ the same as in (iii) when $ s \in \mathbb{N}, 1 \leqslant q < \infty $.
\end{definition}
\medskip

In general, if we assume the solvability of the PDE
\begin{equation*}
Lu = f \; {\rm in} \; M
\end{equation*}
for some second order elliptic operator $ L $, the standard $ \calL^{p} $-type estimates says
\begin{equation*}
\lVert u \rVert_{W^{2,p}(M, g)} \leqslant C \left( \lVert f \rVert_{\calL^{p}(M, g)} + \lVert u \rVert_{\calL^{p}(M, g)} \right).
\end{equation*}
In order to estimate $ \lVert u \rVert_{\calL^{\infty}} $ and $ \lVert \nabla u \rVert_{\calL^{\infty}} $, we would like to remove the term $ \lVert u \rVert_{\calL^{p}(M, g)} $ on the right side of the estimates. The next two results show that sometimes we can do this.
\medskip

\begin{theorem}\label{pre:thm1}\cite[\S7]{Niren4}
Let $ (M, g) $ be a closed manifold and $ q > n = \dim M $ be a given constant. Let $ L: \calC^{\infty}(M) \rightarrow \calC^{\infty}(M) $ be a uniform second order elliptic operator on $ M $ and can be extended to $ L : W^{2, q}(M, g) \rightarrow \calL^{q}(M, g) $. Let $ f \in \calL^{q}(M, g) $ be a given function. Let $ u \in H^{1}(M, g) $ be a weak solution of the following linear PDE
\begin{equation}\label{pre:eqn1}
L u = f \; {\rm in} \; M.
\end{equation}
Assume that $ \text{Ker}(L) = \lbrace 0 \rbrace $. If, in addition, $ u \in \calL^{q}(M, g) $, then $ u \in W^{2, q}(M, g) $ with the following estimates
\begin{equation}\label{pre:eqn2}
\lVert u \rVert_{W^{2, q}(M, g)} \leqslant \gamma \lVert Lu \rVert_{\calL^{q}(M, g)}
\end{equation}
Here $ \gamma $ depends on $ L, q $ and the manifold $ (M, g) $ and is independent of $ u $.
\end{theorem}

\begin{theorem}\label{pre:thm2}\cite[Thm.~2.2]{XU5} Let $ (\bar{M}, g) $ be a compact manifold with smooth boundary $ \partial M $. Let $ \nu $ be the unit outward normal vector along $ \partial M $ and $ q > n = \dim \bar{M} $. Let $ L: \calC^{\infty}(\bar{M}) \rightarrow \calC^{\infty}(\bar{M}) $ be a uniform second order elliptic operator on $ M $ with smooth coefficients up to $ \partial M $ and can be extended to $ L : W^{2, q}(M, g) \rightarrow \calL^{q}(M, g) $. Let $ f \in \calL^{q}(M, g), \tilde{f} \in W^{1, q}(M, g) $. Let $ u \in H^{1}(M, g) $ be a weak solution of the following boundary value problem
\begin{equation}\label{pre:eqn3}
L u = f \; {\rm in} \; M, Bu = \frac{\partial u}{\partial \nu} + c(x) u = \tilde{f} \; {\rm on} \; \partial M.
\end{equation}
Here $ c \in \calC^{\infty}(M) $. Assume also that $ \text{Ker}(L) = \lbrace 0 \rbrace $ associated with the homogeneous Robin boundary condition. If, in addition, $ u \in \calL^{q}(M, g) $, then $ u \in W^{2, q}(M, g) $ with the following estimates
\begin{equation}\label{pre:eqn4}
\lVert u \rVert_{W^{2, q}(M, g)} \leqslant \gamma' \left(\lVert Lu \rVert_{\calL^{q}(M, g)} + \lVert Bu \rVert_{W^{1, q}(M, g)} \right)
\end{equation}
Here $ \gamma' $ depends on $ L, q, c $ and the manifold $ (\bar{M}, g) $ and is independent of $ u $.
\end{theorem}
\begin{remark}\label{pre:re1}
According to the results of Theorem \ref{pre:thm1} and Theorem \ref{pre:thm2}, we have
\begin{equation}\label{pre:eqn5}
\begin{split}
& \frac{1}{q} - \frac{1}{n} \leqslant -\frac{\alpha}{n} \Rightarrow \lVert u \rVert_{\calC^{1, \alpha}(M)} \leqslant K \lVert u \rVert_{W^{2, q}(M, g)}; \\
& \frac{1}{q} - \frac{1}{n} \leqslant -\frac{\alpha}{n} \Rightarrow \lVert u \rVert_{\calC^{1, \alpha}(\bar{M})} \leqslant K' \lVert u \rVert_{W^{2, q}(M, g)},
\end{split}
\end{equation}
due to the H\"older-type Schauder estimates. Here the constant $ K, K' $ depend only on $ q, n, \alpha $, the manifolds $ (M, g) $ or $ (\bar{M}, g) $ and is independent of $ u $. The estimates (\ref{pre:eqn5}) give control of the $ \calL^{\infty} $-norms of $ u $ and $ \nabla u $, respectively.
\end{remark}
\medskip

Our local to global analysis in solving Yamabe-type problems construct lower and upper solutions of the Yamabe-type equation, we then apply the monotone iteration schemes to obtain the solution of theses Yamabe-type equations. In particular, we need the following two versions of the monotone iteration schemes, one for closed manifolds, the other for compact manifolds with non-empty boundary. 

For closed manifolds, we have
\begin{theorem}\label{pre:thm3}\cite[Lemma~2.6]{KW}\cite[Thm.~2.5]{XU3} Let $ (M, g) $ be a closed manifold with $ \dim M \geqslant 3 $. Let $ h, H \in \calC^{\infty}(M) $ for some $ p > n = \dim M $. Let $ m > 1 $ be a constant. If there exists functions $ u_{-}, u_{+} \in \calC_{0}(M) \cap H^{1}(M, g) $ such that
\begin{equation}\label{pre:eqn6}
\begin{split}
-a\Delta_{g} u_{-} + hu_{-} & \leqslant Hu_{-}^{m} \; {\rm in} \; (M, g); \\
-a\Delta_{g} u_{+} + hu_{+} & \geqslant Hu_{+}^{m} \; {\rm in} \; (M, g),
\end{split}
\end{equation}
hold weakly, with $ 0 \leqslant u_{-} \leqslant u_{+} $ and $ u_{-} \not\equiv 0 $, then there is a $ u \in W^{2, p}(M, g) $ satisfying
\begin{equation}\label{pre:eqn7}
-a\Delta_{g} u + hu = Hu^{m} \; {\rm in} \; (M, g).
\end{equation}
In particular, $ u \in \calC^{\infty}(M) $.
\end{theorem}
For compact manifolds with non-empty boundary, we now introduce a variation of the monotone iteration scheme we used in \cite{XU4}, \cite{XU5} and \cite{XU6}. In particular, we do require $ h_{g} = h \geqslant 0 $ to be some positive constant on $ \partial M $ here, this can be done due to the proof of the Han-Li conjecture in \cite{XU5}. We point out that the proof of the result below is similar to Theorem 4.1 in \cite{XU5}, but technically more subtle.
\begin{theorem}\label{pre:thm4}
Let $ (\bar{M}, g) $ be a compact manifold with smooth boundary $ \partial M $. Let $ \nu $ be the unit outward normal vector along $ \partial M $ and $ q > \dim \bar{M} $. Let $ S \in \calC^{\infty}(\bar{M}) $ and $ H \in \calC^{\infty}(\bar{M}) $ be given functions. Let the mean curvature $ h_{g} = h \geqslant 0 $ be some positive constant. In addition, we assume that $ \sup_{\bar{M}} \lvert H \rvert $ is small enough. Suppose that there exist $ u_{-} \in \calC_{0}(\bar{M}) \cap H^{1}(M, g) $ and $ u_{+} \in W^{2, q}(M, g) \cap \calC_{0}(\bar{M}) $, $ 0 \leqslant u_{-} \leqslant u_{+} $, $ u_{-} \not\equiv 0 $ on $ \bar{M} $, some constants $ \theta_{1} \leqslant 0, \theta_{2} \geqslant 0 $ such that
\begin{equation}\label{pre:eqn8}
\begin{split}
-a\Delta_{g} u_{-} + R_{g} u_{-} - S u_{-}^{p-1} & \leqslant 0 \; {\rm in} \; M, \frac{\partial u_{-}}{\partial \nu} + \frac{2}{p-2} h_{g} u_{-} \leqslant \theta_{1} u_{-} \leqslant \frac{2}{p-2} H u_{-}^{\frac{p}{2}} \; {\rm on} \; \partial M \\
-a\Delta_{g} u_{+} + R_{g} u_{+} - S u_{+}^{p-1} & \geqslant 0 \; {\rm in} \; M, \frac{\partial u_{+}}{\partial \nu} + \frac{2}{p-2} h_{g} u_{+} \geqslant \theta_{2} u_{+} \geqslant \frac{2}{p-2} H u_{+}^{\frac{p}{2}} \; {\rm on} \; \partial M
\end{split}
\end{equation}
holds weakly. In particular, $ \theta_{1} $ can be zero if $ H \geqslant 0 $ on $ \partial M $, and $ \theta_{1} $ must be negative if $ H < 0 $ somewhere on $ \partial M $; similarly, $ \theta_{2} $ can be zero if $ H \leqslant 0 $ on $ \partial M $, and $ \theta_{2} $ must be positive if $ H > 0 $ somewhere on $ \partial M $. Then there exists a real, positive solution $ u \in \calC^{\infty}(M) \cap \calC^{1, \alpha}(\bar{M}) $ of
\begin{equation}\label{pre:eqn9}
\Box_{g} u = -a\Delta_{g} u + R_{g} u = S u^{p-1}  \; {\rm in} \; M, B_{g} u =  \frac{\partial u}{\partial \nu} + \frac{2}{p-2} h_{g} u = \frac{2}{p-2} H u^{\frac{p}{2}} \; {\rm on} \; \partial M.
\end{equation}
\end{theorem}
\begin{proof} From now on, we replace $ h_{g} $ by the constant $ h $ due to the hypothesis. Fix some $ q > \dim \bar{M} $. Denote $ u_{0} = u_{+} $. Due to compactness of $ \bar{M} $, we can choose a constant $ A > 0 $ such that
\begin{equation}\label{pre:eqn10}
-R_{g}(x) + S(x) (p - 1) u(x)^{p-2} + A > 0, \forall u(x) \in [\min_{\bar{M}} u_{-}(x), \max_{\bar{M}} u_{+}(x) ], \forall x \in \bar{M}
\end{equation}
pointwise. 
Similarly, we can also choose a constant $ B \geqslant 0 $ such that
\begin{equation}\label{pre:eqn11}
-\frac{2}{p-2} h + \frac{p}{p-2} H(y) u(y)^{\frac{p-2}{2}} + B > 0, \forall u(y) \in [\min_{\bar{M}} u_{-}(y), \max_{\bar{M}} u_{+}(y) ], \forall y \in \partial M.
\end{equation}
For the first step, consider the linear PDE
\begin{equation}\label{pre:eqn12}
-a\Delta_{g} u_{1} + Au_{1} = Au_{0} - R_{g} u_{0} + S u_{0}^{p-1}  \; {\rm in} \; M, \frac{\partial u_{1}}{\partial \nu} + Bu_{1} = Bu_{0} - \frac{2}{p-2} h u_{0} + \frac{2}{p-2} H u_{0}^{\frac{p}{2}} \; {\rm on} \; \partial M.
\end{equation}
Since $ u_{0} = u_{+} \in W^{2, q}(M, g) \cap \calC_{0}(\bar{M}) $, there exists a unique solution $ u_{1} \in H^{1}(M, g) $, due to the fact that $ A > 0 $ and $ B \geqslant 0 $ and thus the standard Lax-Milgram theorem applies. Since $ u_{0} \in W^{2, q}(M, g) \cap \calC_{0}(\bar{M}) $, and thus $ u_{0} \in \calC^{1, \alpha}(\bar{M}) \cap \calL^{q}(M, g) $ for all $ 1 < q < \infty $, it follows from $ \calL^{p}$-regularity in Theorem 2.1 of \cite{XU5} that $ u_{1} \in W^{2, q}(M, g) $. By standard $ (s, p) $-type Sobolev embedding, it follows that $ u_{1} \in \calC^{1, \alpha}(\bar{M}) $ for some $ \alpha \in (0, 1) $. 

Next we show that $ u_{1} \leqslant u_{0} = u_{+} $.
Subtracting the second equation in (\ref{pre:eqn8}) by (\ref{pre:eqn12}), we have
\begin{equation*}
\left( -a\Delta_{g} + A \right) (u_{0} - u_{1}) \geqslant 0 \; {\rm in} \; M, \left( \frac{\partial}{\partial \nu} + B \right) (u_{0} - u_{1}) \geqslant 0 \; {\rm on} \; \partial M.
\end{equation*}
in the weak sense, due to the choices of $ A $ and $ B $ in (\ref{pre:eqn10}) and (\ref{pre:eqn11}). Denote
\begin{equation*}
w = \max \lbrace 0, u_{1} - u_{0} \rbrace.
\end{equation*}
It is immediate that $ w \in H^{1}(M, g) \cap \calC_{0}(\bar{M}) $ and $ w \geqslant 0 $. It follows that
\begin{align*}
0 & \geqslant \int_{M} \left( a \nabla_{g} (u_{1} - u_{0}) \cdot \nabla_{g} w + A(u_{1} - u_{0}) w \right) d\omega + \int_{\partial M} B (u_{1} - u_{0}) w dS \\
& = \int_{M} \left( a \lvert \nabla_{g} w \rvert^{2} + A w^{2} \right) d\omega + \int_{\partial M} B w^{2} dS \geqslant 0.
\end{align*}
The last inequality holds since $ A > 0 $ and $ B \geqslant 0 $. It follows that
\begin{equation*}
w \equiv 0 \Rightarrow 0 \geqslant u_{1} - u_{0} \Rightarrow u_{0} \geqslant u_{1}.
\end{equation*}
By the same argument, we can show that $ u_{1} \geqslant u_{-} $ and hence $ u_{-} \leqslant u_{1} \leqslant u_{+} $. Assume inductively that $ u_{-} \leqslant \dotso \leqslant u_{k-1} \leqslant u_{k} \leqslant u_{+} $ for some $ k > 1 $ with $ u_{k} \in W^{2, q}(M, g) $ , the $ (k + 1)th $ iteration step is
\begin{equation}\label{pre:eqn13}
\begin{split}
-a\Delta_{g} u_{k+1} + Au_{k+1} & = Au_{k} - R_{g} u_{k} + S u_{k}^{p-1}  \; {\rm in} \; M, \\
 \frac{\partial u_{k+1}}{\partial \nu} +Bu_{k + 1} & = Bu_{k} - \frac{2}{p-2} h u_{k} + \frac{2}{p-2} H u_{k}^{\frac{p}{2}} \; {\rm on} \; \partial M.
\end{split}
\end{equation}
Since $ u_{k} \in W^{2, q}(M, g) $ thus $ u_{k} \in \calC^{1, \alpha}(\bar{M}) $ due to Sobolev embedding, by the same reason as the first step, we conclude that there exists $ u_{k+1} \in W^{2, q}(M, g) $ that solves (\ref{pre:eqn13}). In particular, $ u_{k}^{\frac{p}{2}} = u_{k}^{\frac{n}{n-2}} \in \calC^{1}(\bar{M}) $ hence the hypothesis of the boundary condition in Theorem 2.1 and Theorem 2.4 of \cite{XU5} are satisfied. 

We show that $ u_{-} \leqslant u_{k + 1} \leqslant u_{k} \leqslant u_{+} $. The $ kth $ iteration step is of the form
\begin{equation}\label{pre:eqn14}
\begin{split}
-a\Delta_{g} u_{k} + Au_{k} & = Au_{k - 1} - R_{g} u_{k - 1} + S u_{k - 1}^{p-1}  \; {\rm in} \; M, \\
\frac{\partial u_{k}}{\partial \nu} + Bu_{k} & = Bu_{k - 1} -  \frac{2}{p-2} h u_{k - 1} + \frac{2}{p-2} H u_{k - 1}^{\frac{p}{2}} \; {\rm on} \; \partial M.
\end{split}
\end{equation}
Subtracting (\ref{pre:eqn13}) by (\ref{pre:eqn14}), we conclude that
\begin{align*}
& \left( -a\Delta_{g} + A \right) \left( u_{k + 1} - u_{k} \right) = A(u_{k} - u_{k - 1}) - R_{g} (u_{k} - u_{k - 1}) + \lambda \left( u_{k}^{p-1} - u_{k - 1}^{p-1} \right) \leqslant 0 \; {\rm in} \; M; \\
& \frac{\partial \left(u_{k+1} - u_{k}\right)}{\partial \nu} + B(u_{k + 1} - u_{k}) = B(u_{k} - u_{k-1}) \\
& \qquad - \frac{2}{p-2} h (u_{k+ 1} - u_{k}) + \frac{2}{p-2} H u_{k}^{\frac{p}{2}}  - \frac{2}{p-2} H u_{k - 1}^{\frac{p}{2}} \leqslant 0 \; {\rm on} \; \partial M.
\end{align*}
By induction we have $ u_{-} \leqslant u_{k} \leqslant u_{k - 1} \leqslant u_{+} $. The first inequality above is then due to the pointwise mean value theorem and the choice of $ A $ in (\ref{pre:eqn10}). Similarly, the second inequality is due to the choice of $ B $ in (\ref{pre:eqn11}). Note that since both $ u_{k}, u_{k-1} \in W^{2, q}(M, g) $, above inequalities hold in strong sense. We choose
\begin{equation*}
\tilde{w} = \max \lbrace 0, u_{k+1} - u_{k} \rbrace.
\end{equation*}
Clearly $ \tilde{w} \geqslant 0 $ with $ \tilde{w} \in H^{1}(M, g) \cap \calC_{0}(\bar{M}) $. Pairing $ \tilde{w} $ with $ \left( -a\Delta_{g} + A \right) \left( u_{k} - u_{k + 1} \right) \leqslant 0 $, we have
\begin{align*}
0 & \geqslant \int_{M} \left( -a\Delta_{g} + A \right) \left( u_{k+1} - u_{k} \right)  \tilde{w} d\omega = \int_{M} a \nabla_{g} (u_{k+1} - u_{k}) \cdot \nabla_{g} \tilde{w} d\omega - \int_{\partial M} \frac{\partial \left(u_{k+1} - u_{k}\right)}{\partial \nu} \tilde{w} dS \\
& \geqslant  \int_{M} a \nabla_{g} (u_{k+1} - u_{k}) \cdot \nabla_{g} \tilde{w} d\omega + \int_{\partial M} B (u_{k+1} - u_{k}) \tilde{w} dS \\
& = a \lVert \nabla_{g} \tilde{w} \rVert_{\calL^{2}(M, g)}^{2} + B \int_{\partial M} \tilde{w}^{2} dS \geqslant 0.
\end{align*}
It follows that
\begin{equation*}
\tilde{w} = 0 \Rightarrow 0 \geqslant u_{k + 1} - u_{k} \Rightarrow u_{k + 1} \leqslant u_{k}.
\end{equation*}
By the same argument and the induction $ u_{k} \geqslant u_{-} $, we conclude that $ u_{k+1} \geqslant u_{-} $. Thus
\begin{equation}\label{pre:eqn15}
0 \leqslant u_{-} \leqslant u_{k+1} \leqslant u_{k} \leqslant u_{+}, u_{k} \in W^{2, q}(M, g), \forall k \in \mathbb{N}.
\end{equation}
To apply the Arzela-Ascoli Theorem, we need to show that $ \lVert u_{k} \rVert_{W^{2,q}(M, g)} $ is uniformly bounded above in $ k $. By Theorem 2.4 of \cite{XU5}, the operator $ -a\Delta_{g}u + Au $ with the homogeneous Robin condition $ \frac{\partial u}{\partial \nu} + B u = 0 $ for $ B \geqslant 0 $ is injective. Applying $ L^{p} $-regularity in Theorem \ref{pre:thm2}, we conclude from the first iteration step (\ref{pre:eqn12}) that
\begin{equation}\label{pre:eqn16}
\lVert u_{1} \rVert_{W^{2, q}(M, g)} \leqslant C' \left( \lVert Au_{0} - R_{g} u_{0} + S u_{0}^{p-1} \rVert_{\calL^{q}(M, g)} + \left\lVert \left( B - \frac{2}{p-2} h \right) u_{0} + \frac{2}{p-2} H u_{0}^{\frac{p}{2}} \right\rVert_{W^{1, q}(M, g)} \right).
\end{equation}
We point out that the constant $ C' $ only depends on the metric $ g $, the differential operators in the interior and on the boundary; in particular, the constant $ C' $ is kept the same if we make $ B $ smaller. Briefly speaking, the $ Bu $ term will be reflected in $ \lVert Bu \rVert_{W^{1, p}(M, g)} $ on the right side; hence the Peter-Paul inequality will allow us to subtract less in terms of $ \lVert u \rVert_{W^{2,p}(M, g)} $ on the left side of the elliptic regularity. Consider the formula (\ref{pre:eqn11}) in which we make the choice of the constant $ B $. Since $ u_{-} $ and $ u_{+} $ are fixed, it follows that the smaller the $ \sup_{\partial M} \lvert H \rvert $, the smaller the constant $ B $ we can choose so that $ B - \frac{2}{p-2} h \rightarrow 0 $.

Choose $ \sup_{\partial M} \lvert H \rvert > 0 $ small enough so that
\begin{equation}\label{pre:eqn17}
\begin{split}
& \left\lVert \left( B - \frac{2}{p-2} h \right) u_{0} + \frac{2}{p-2} H u_{0}^{\frac{p}{2}} \right\rVert_{W^{1, q}(M, g)}\leqslant 1; \\
& \frac{2}{p-2}   \sup_{\bar{M}} \left( u_{0}^{\frac{p}{2}} \right) \cdot Vol_{g}(\bar{M}) \left( \sup_{\bar{M}} \lvert H \rvert +  \sup_{\partial M} \lvert \nabla H \rvert^{p} \right) \\
& \qquad +\left( B - \frac{2}{p-2}h + \frac{p}{p-2} \sup_{\bar{M}} \lvert H \rvert \sup_{\bar{M}} \left( u_{0}^{\frac{p - 2}{2}} \right) \right) \cdot \\
& \qquad \qquad \cdot C' \left( \left( A  + \sup_{\bar{M}} \lvert R_{g} \rvert + \sup_{\bar{M}} \lvert S \rvert \sup_{\bar{M}} \left( u_{0}^{p-2} \right) \right) \sup_{\bar{M}} \left( u_{0} \right) \cdot \text{Vol}_{g}(\bar{M}) + 1 \right) \leqslant 1.
\end{split}
\end{equation}
Note that for smaller $ \sup_{\bar{M}} \lvert H \rvert $, the sub-solution and supers-olution in (\ref{pre:eqn8}) still hold, due to our hypotheses of $ \theta_{1} $ and $ \theta_{2} $. Note also that the choice of $ H $ in terms of (\ref{pre:eqn17}) does not depend on $ k $. Due to (\ref{pre:eqn17}), we conclude that
\begin{align*}
\lVert u_{1} \rVert_{W^{2, q}(M, g)} & \leqslant C' \left( \lVert Au_{0} - R_{g} u_{0} +S u_{0}^{p-1} \rVert_{\calL^{q}(M, g)} + 1 \right) \\
& \leqslant C' \left( \left( A  + \sup_{\bar{M}} \lvert R_{g} \rvert + \sup_{\bar{M}} \lvert S \rvert \sup_{\bar{M}} \left( u_{0}^{p-2} \right) \right) \sup_{\bar{M}} \left( u_{0} \right) \cdot \text{Vol}_{g}(\bar{M}) + 1 \right).
\end{align*}
Inductively, we assume
\begin{equation}\label{pre:eqn18}
\lVert u_{k} \rVert_{W^{2, q}(M, g)} \leqslant C' \left( \left( A  + \sup_{\bar{M}} \lvert R_{g} \rvert + \sup_{\bar{M}} \lvert S \rvert \sup_{\bar{M}} \left( u_{0}^{p-2} \right) \right) \sup_{\bar{M}} \left( u_{0} \right) \cdot \text{Vol}_{g}(\bar{M}) + 1 \right).
\end{equation}
For $ u_{k + 1} $, we conclude from (\ref{pre:eqn13}) that
\begin{equation}\label{pre:eqn19}
\begin{split}
\lVert u_{k+1} \rVert_{W^{2, q}(M, g)} & \leqslant C' \lVert Au_{k} - R_{g} u_{k} + \lambda u_{k}^{p-1} \rVert_{\calL^{q}(M, g)} \\
& \qquad + C' \left\lVert \left( B - \frac{2}{p-2} h \right)u_{0} \right\rVert_{W^{1, q}(M, g)} + C' \left\lVert \frac{2}{p-2} \zeta u_{k}^{\frac{p}{2}} \right\rVert_{W^{1, q}(M, g)}.
\end{split}
\end{equation}
The last term in (\ref{pre:eqn19}) can be estimated as
\begin{align*}
\left\lVert \frac{2}{p-2} H u_{k}^{\frac{p}{2}} \right\rVert_{W^{1, q}(M, g)} & = \frac{2}{p-2} \left( \lVert H u_{k}^{\frac{p}{2}} \rVert_{\calL^{q}(M, g)} + \left\lVert \nabla_{g} \left(H \left( u_{k}^{\frac{p}{2}} \right) \right) \right\rVert_{\calL^{q}(M, g)} \right) \\
& \leqslant  \frac{2}{p-2} \sup_{\bar{M}} \lvert H \rvert \sup_{\bar{M}} \left( u_{k}^{\frac{p}{2}} \right) \cdot Vol_{g}(\bar{M}) \\
& \qquad + \frac{p}{p-2} \sup_{\bar{M}} \lvert H \rvert \sup_{\bar{M}} \left( u_{k}^{\frac{p - 2}{2}} \right) \lVert \nabla_{g} u_{k}  \rVert_{\calL^{q}(M, g)} \\
& \qquad \qquad + \frac{2}{p-2} \sup_{\bar{M}} \left( u_{k}^{\frac{p}{2}} \right) \sup_{\partial M} \lvert \nabla H \rvert^{p} \cdot Vol_{g}(\bar{M}) \\
& \leqslant  \frac{2}{p-2}   \sup_{\bar{M}} \left( u_{0}^{\frac{p}{2}} \right) \cdot Vol_{g}(\bar{M}) \left( \sup_{\bar{M}} \lvert H \rvert +  \sup_{\partial M} \lvert \nabla H \rvert^{p} \right) \\
& \qquad +\frac{p}{p-2}  \sup_{\bar{M}} \lvert H \rvert \cdot \sup_{\bar{M}}  \left( u_{0}^{\frac{p - 2}{2}} \right) \lVert u_{k}  \rVert_{W^{2, q}(M, g)}.
\end{align*}
By the choice of $ H $ in (\ref{pre:eqn17}) and induction assumption in (\ref{pre:eqn18}), we conclude that
\begin{equation*}
\left\lVert \left( B - \frac{2}{p-2} h \right)u_{k} \right\rVert_{W^{1, q}(M, g)} + \left\lVert \frac{2}{p-2} \zeta u_{k}^{\frac{p}{2}} \right\rVert_{W^{1, q}(M, g)} \leqslant 1.
\end{equation*}
It follows from (\ref{pre:eqn19}) that
\begin{equation}\label{pre:eqn20}
\begin{split}
\lVert u_{k+1} \rVert_{W^{2, q}(M, g)} & \leqslant C' \lVert Au_{k} - R_{g} u_{k} + \lambda u_{k}^{p-1} \rVert_{\calL^{q}(M, g)} \\
& \qquad + C' \left\lVert \left( B - \frac{2}{p-2} h \right)u_{0} \right\rVert_{W^{1, q}(M, g)} + C' \left\lVert \frac{2}{p-2} \zeta u_{k}^{\frac{p}{2}} \right\rVert_{W^{1, q}(M, g)} \\
& \leqslant C' \left( \left( A  + \sup_{\bar{M}} \lvert R_{g} \rvert + \lvert \lambda \rvert \sup_{\bar{M}} \left( u_{0}^{p-2} \right) \right) \sup_{\bar{M}} \left( u_{0} \right) \cdot \text{Vol}_{g}(\bar{M}) + 1 \right).
\end{split}
\end{equation}
It follows that the sequence $ \lbrace u_{k} \rbrace_{k \in \mathbb{N}} $ is uniformly bounded in $ W^{2, q} $-norm. By Sobolev embedding, we conclude that the same sequence is uniformly bounded in $ \calC^{1, \alpha} $-norm with some $ \alpha \in (0, 1) $. Thus by Arzela-Ascoli theorem, we conclude that there exists $ u $ such that
\begin{equation*}
u = \lim_{k \rightarrow \infty} u_{k}, 0 \leqslant u_{-} \leqslant u \leqslant u_{+}, \Box_{g} u = \lambda u^{p-1}  \; {\rm in} \; M, B_{g} u = \frac{2}{p-2} \zeta u^{\frac{p}{2}} \; {\rm on} \; \partial M
\end{equation*}
in the strong sense. Apply the elliptic regularity, we conclude that $ u \in W^{2, q}(M, g) $. A standard bootstrapping argument concludes that $ u \in \calC^{\infty}(M) \cap \calC^{1, \alpha}(\bar{M}) $, due to Schauder estimates. The regularity of $ u $ on $ \partial M $ is determined by $ u^{p-1} $. We then apply the trace theorem \cite[Prop.~4.5]{T} to show that $ u $ is also smooth on $ \partial M $.

Lastly we show that $ u $ is positive. Since $ u \in \calC^{\infty}(M) $ it is smooth locally, the local strong maximum principle says that if $ u = 0 $ in some interior domain $ \Omega $ then $ u \equiv 0 $ on $ \Omega $, a continuation argument then shows that $ u \equiv 0 $ in $ M $. But $ u \geqslant u_{-} $ and $ u_{-} > 0 $ within some region. Thus $ u > 0 $ in the interior $ M $. By the same argument in \cite[\S1]{ESC}, we conclude that $ u > 0 $ on $ \bar{M} $.
\end{proof}
We have an immediate consequence of Theorem \ref{pre:thm4} for a perturbed conformal Laplacian operator.
\begin{corollary}\label{pre:cor1}
Let $ (\bar{M}, g) $ be a compact manifold with smooth boundary $ \partial M $. Let $ \nu $ be the unit outward normal vector along $ \partial M $ and $ q > \dim \bar{M} $. Let $ S \in \calC^{\infty}(\bar{M}) $ and $ H \in \calC^{\infty}(\bar{M}) $ be given functions. Let the mean curvature $ h_{g} = h \geqslant 0 $ be some positive constant and $ \beta < 0 $ be some negative constant. In addition, we assume that $ \sup_{\bar{M}} \lvert H \rvert $ is small enough. Suppose that there exist $ u_{-} \in \calC_{0}(\bar{M}) \cap H^{1}(M, g) $ and $ u_{+} \in W^{2, q}(M, g) \cap \calC_{0}(\bar{M}) $, $ 0 \leqslant u_{-} \leqslant u_{+} $, $ u_{-} \not\equiv 0 $ on $ \bar{M} $, some constants $ \theta_{1} \leqslant 0, \theta_{2} \geqslant 0 $ such that
\begin{equation}\label{pre:eqn21}
\begin{split}
-a\Delta_{g} u_{-} + \left(R_{g} + \beta \right) u_{-} - S u_{-}^{p-1} & \leqslant 0 \; {\rm in} \; M, \frac{\partial u_{-}}{\partial \nu} + \frac{2}{p-2} h_{g} u_{-} \leqslant \theta_{1} u_{-} \leqslant \frac{2}{p-2} H u_{-}^{\frac{p}{2}} \; {\rm on} \; \partial M \\
-a\Delta_{g} u_{+} + \left(R_{g} + \beta \right) u_{+} - S u_{+}^{p-1} & \geqslant 0 \; {\rm in} \; M, \frac{\partial u_{+}}{\partial \nu} + \frac{2}{p-2} h_{g} u_{+} \geqslant \theta_{2} u_{+} \geqslant \frac{2}{p-2} H u_{+}^{\frac{p}{2}} \; {\rm on} \; \partial M
\end{split}
\end{equation}
holds weakly. In particular, $ \theta_{1} $ can be zero if $ H \geqslant 0 $ on $ \partial M $, and $ \theta_{1} $ must be negative if $ H < 0 $ somewhere on $ \partial M $; similarly, $ \theta_{2} $ can be zero if $ H \leqslant 0 $ on $ \partial M $, and $ \theta_{2} $ must be positive if $ H > 0 $ somewhere on $ \partial M $. Then there exists a real, positive solution $ u \in \calC^{\infty}(M) \cap \calC^{1, \alpha}(\bar{M}) $ of
\begin{equation}\label{pre:eqn22}
\Box_{g} u = -a\Delta_{g} u +\left( R_{g} + \beta \right) u = S u^{p-1}  \; {\rm in} \; M, B_{g} u =  \frac{\partial u}{\partial \nu} + \frac{2}{p-2} h_{g} u = \frac{2}{p-2} H u^{\frac{p}{2}} \; {\rm on} \; \partial M.
\end{equation}
\end{corollary}
\begin{proof} By replacing $ R_{g} $ by $ R_{g} + \beta $, everything is essentially the same as in the proof of Theorem \ref{pre:thm4}.
\end{proof}
\medskip

\section{The Local Analysis on Small Riemannian Domains}
Another key step in our local to global analysis is the existence of some positive, smooth solution of the following local Yamabe equation with Dirichlet boundary condition,
\begin{equation}\label{local:eqn1}
-a\Delta_{g} u + R_{g} u = f u^{p-1} \; {\rm in} \; \Omega, u \equiv 0 \; {\rm on} \; \partial \Omega.
\end{equation}
Here $ f $ is some positive, smooth function defined on a neighborhood of $ \Omega $. We have shown results for positive constant function $ f = \lambda > 0 $, especially in \cite[Prop.~2.4]{XU6} and \cite[Prop.~2.5]{XU6}. We would like to point out that when $ f $ is a positive constant within some open region and the dimension is at least $ 6 $, then a simpler method can be applied, essentially due to Aubin's local test function for the Yamabe problem. We would like to revisit the constant function case below for this simple case, and then show why the new method developed in \cite{XU6} is essential for the non-constant positive functions $ f $.

Although the methods when $ f $ is not a constant is quite similar to the cases in \cite{XU6}, some subtle technicality forces us to provide all details here. We now discuss the case when the manifold is not locally conformally flat, the analysis is essentially due to the argument in the previous papers, see \cite{XU4}, \cite{XU5}, \cite{XU6} and \cite{XU3}.

Locally we treat (\ref{local:eqn1}) as the general second order linear elliptic PDE with the Dirichlet boundary condition:
\begin{equation}\label{local:eqn2}
\begin{split}
Lu & : = -\sum_{i, j} \partial_{i} \left (a_{ij}(x) \partial_{j} u \right) = b(x) u^{p- 1} + f(x, u) \; {\rm in} \; \Omega; \\
u & > 0 \; {\rm in} \; \Omega, u = 0 \; {\rm on} \; \partial \Omega.
\end{split}
\end{equation}
Recall that $ p - 1 = \frac{n+2}{n -2} $ is the critical exponent with respect to the $ H_{0}^{1} $-solutions of (\ref{local:eqn2}) in the sense of Sobolev embedding. Due to variational method, (\ref{local:eqn2}) is the Euler-Lagrange equation of the functional
\begin{equation}\label{local:eqn3}
J(u) = \int_{\Omega} \left( \frac{1}{2} \sum_{i, j} a_{ij}(x) \partial_{i}u \partial_{j} u - \frac{b(x)}{p} u_{+}^{p} - F(x, u) \right) dx,
\end{equation}
with appropriate choices of $ a_{ij}, b $ and $ F $. Here $ u_{+} = \max \lbrace u, 0 \rbrace $ and $ F(x, u) = \int_{0}^{u} f(x, t)dt $. Set
\begin{equation}\label{local:eqn4}
\begin{split}
A(O) & = \text{essinf}_{x \in O} \frac{\det(a_{ij}(x))}{\lvert b(x) \rvert^{n-2}}, \forall O \subset \Omega; \\
T & = \inf_{u \in H_{0}^{1}(\Omega)}  \frac{\int_{\Omega} \lvert Du \rvert^{2} dx}{\left( \int_{\Omega} \lvert u \rvert^{p} dx \right)^{\frac{2}{p}}}; \\
K & = \inf_{u \neq 0} \sup_{t > 0} J(tu), K_{0} = \frac{1}{n} T^{\frac{n}{2}} \left( A(\Omega) \right)^{\frac{1}{2}}.
\end{split}
\end{equation}
The core theorem we need to use is due to Wang \cite[Thm.~1.1]{WANG}. 
\begin{theorem}\label{local:thm1}\cite[Thm.~1.1, Thm.~1.4]{WANG} Let $ \Omega $ be a bounded smooth domain in $ \R^{n}, n \geqslant 3 $. Let $ Lu = -\sum_{i, j} \partial_{i} \left (a_{ij}(x) \partial_{j} u \right) $ be a second order elliptic operator with smooth coefficients in divergence form. Let ${\rm Vol}_g(\Omega)$ and the diameter of $\Omega$ sufficiently small. Let $ b(x) \neq 0 $ be a nonnegative bounded measurable function. Let $ f(x, u) $ be measurable in $ x $ and continuous in $ u $. Assume
\begin{enumerate}[(P1).]
\item There exist $ c_{1}, c_{2} > 0 $ such that $ c_{1} \lvert \xi \rvert^{2} \leqslant \sum_{i, j} a_{ij}(x) \xi_{i} \xi_{j} \leqslant c_{2} \lvert \xi \rvert^{2}, \forall x \in \Omega, \xi \in \R^{n} $;
\item $ \lim_{u \rightarrow + \infty} \frac{f(x, u)}{u^{p-1}} = 0 $ uniformly for $ x \in \Omega $;
\item $ \lim_{u \rightarrow 0} \frac{f(x, u)}{u} < \lambda_{1} $ uniformly for $ x \in \Omega $, where $ \lambda_{1} $ is the first eigenvalue of $ L $;
\item There exists $ \theta \in (0, \frac{1}{2}), M \geqslant 0, \sigma > 0 $, such that $ F(x, u) = \int_{0}^{u} f(x, t)dt \leqslant \theta u f(x, u) $ for any $ u \geqslant M $, $ x \in \Omega(\sigma) = \lbrace x \in \Omega, 0 \leqslant b(x) \leqslant \sigma \rbrace $.
\end{enumerate}
Furthermore, we assume that $ f(x, u) \geqslant 0 $, $ f(x, u) = 0 $ for $ u \leqslant 0 $. We also assume that $ a_{ij}(x) \in \calC^{0}(\bar{\Omega}) $. If
\begin{equation}\label{local:eqn5}
K < K_{0}
\end{equation}
then the Dirichlet problem (\ref{local:eqn2}) possesses a positive solution $ u \in \calC^{\infty}(\Omega) \cap \calC^{0}(\bar{\Omega}) $ which satisfies $ J(u) \leqslant K $.
\end{theorem}
\medskip

When the manifold is not locally conformally flat, and the dimension is at least 6, we apply Aubin's local test result in terms of conformal normal coordinates \cite{PL} to get our first existence result of (\ref{local:eqn1}), provided that the function $ f $ is a positive constant within the domain.
\begin{proposition}\label{local:prop1}
Let $ (\Omega, g) $ be a Riemannian ball in $\R^n$ with $C^{\infty} $ boundary, $ n \geqslant 6 $, centered at a point $ P $ such that the Weyl tensor at $ P $ does not vanish. Let $ f = \lambda > 0 $ be a positive constant function in $ \Omega $. Assume that ${\rm Vol}_g(\Omega)$ and the Euclidean diameter of $\Omega$ sufficiently small. In addition, we assume that the first eigenvalue of Laplace-Beltrami operator $ -\Delta_{g} $ on $ \Omega $ with Dirichlet condition satisfies $ \lambda_{1} \rightarrow \infty $ as $ \Omega $ shrinks. If $ R_{g} < 0 $ within the small enough closed domain $ \bar{\Omega} $, then the Dirichlet problem (\ref{local:eqn1}) has a real, positive, smooth solution $ u \in \calC^{\infty}(\Omega) \cap H_{0}^{1}(\Omega, g) \cap \calC^{0}(\bar{\Omega}) $.
\end{proposition}
\begin{proof}
We apply Theorem \ref{local:thm1} to get the existence result. As shown in \cite[Prop.~3.3]{XU3}, the hypothesis (P1) through (P4) are satisfied. The key is to show that the inequality (\ref{local:eqn5}) holds. By \cite[Prop.~2.2]{XU6} and \cite[Prop.3.3]{XU3}, showing $ K < K_{0} $ is equivalent to show that there exists a positive test function $ u \in \calC_{c}^{\infty}(\Omega) $ such that 
\begin{equation}\label{local:eqn6}
J_{0} : = \frac{\int_{\Omega} \sqrt{\det(g)} g^{ij} \partial_{i} u \partial_{j} u dx + \frac{1}{a} \int_{\Omega}  \sqrt{\det(g)} R_{g} u^{2} dx}{\left( \int_{\Omega} \sqrt{\det(g)} u^{p} dx \right)^{\frac{2}{p}}} <  T.
\end{equation}
According to the Aubin's choice of test function $ u = \frac{\varphi_{\Omega}(x)}{\left( \epsilon + \lvert x \rvert^{2} \right)^{\frac{n-2}{2}}} $ with the cut-off function $ \varphi_{\Omega} \equiv 1 $ in a small ball of $ P $, Aubin showed that
\begin{equation*}
J_{0} \leqslant \Lambda < T.
\end{equation*}
Here $ \Lambda $ is a positive constant that only depends on the evaluation of the Weyl tensor at the point $ P $ and the constant $ \epsilon $ in the test function, see e.g. the proof of Theorem B in \cite{PL}. Thus all hypothesis in Theorem \ref{local:thm1} hold. It follows that (\ref{local:eqn1}) has a smooth, positive solution within a small enough domain $ \Omega $. The regularity argument follows exactly the same as in \cite[Prop.~2.2]{XU6} and \cite[Prop.3.3]{XU3}.
\end{proof}
\begin{remark}\label{local:re1} If $ f $ is not a constant function, then (\ref{local:eqn6}) becomes
\begin{equation*}
\frac{\int_{\Omega} \sqrt{\det(g)} g^{ij} \partial_{i} u \partial_{j} u dx + \frac{1}{a} \int_{\Omega}  \sqrt{\det(g)} R_{g} u^{2} dx}{\left( \int_{\Omega} \sqrt{\det(g)} f u^{p} dx \right)^{\frac{2}{p}}} <  \left( \max_{\bar{\Omega}} (f) \right)^{\frac{2 - n}{n}} T.
\end{equation*}
We must show that
\begin{equation}\label{local:eqn7}
\frac{\int_{\Omega} \sqrt{\det(g)} g^{ij} \partial_{i} u \partial_{j} u dx + \frac{1}{a} \int_{\Omega}  \sqrt{\det(g)} R_{g} u^{2} dx}{\left( \int_{\Omega} \sqrt{\det(g)} u^{p} dx \right)^{\frac{2}{p}}} <  \left( \frac{\min_{\bar{\Omega}} (f)}{\max_{\bar{\Omega}} (f)} \right)^{\frac{n - 2}{n}} T.
\end{equation}
When we shrink the size of the domain to take $ \left( \frac{\min_{\bar{\Omega}} f}{\max_{\bar{\Omega}} f} \right)^{\frac{n - 2}{n}} \rightarrow 1 $, the choice of $ \epsilon $ will be smaller, since it depends on the size of the domain. Thus it might take the threshold $ \Lambda $ to be closer to $ T $. It is not clear how we can get the inequality (\ref{local:eqn7}) unless we put restrictions on the function $ f $.
\end{remark}
\medskip

Due to the difficulty in Remark \ref{local:re1}, we cannot use Aubin's test function directly, neither the local Yamabe equation. We turn to the perturbed local Yamabe equation
\begin{equation}\label{local:eqn8}
-a\Delta_{g} u + \left( R_{g} + \beta \right) u = f u^{p-1} \; {\rm in} \; \Omega, u = 0 \; {\rm on} \; \partial \Omega
\end{equation}
for some constant $ \beta < 0 $ and modify the analysis in \cite[\S2]{XU6}. The goal is to gain a uniform gap between $ K $ and $ K_{0} $ as given in (\ref{local:eqn4}) with respect to the perturbed local Yamabe equation. Then the limiting argument allows use to consider the sequence of solutions of (\ref{local:eqn8}) when $ \beta \rightarrow 0^{-} $.
\begin{proposition}\label{local:prop2}
Let $ (\Omega, g) $ be a not locally conformally flat Riemannian domain in $\R^n$ with $C^{\infty} $ boundary, $ n \geqslant 3 $. Let $ f \in \Omega' \supset \Omega $ be a positive, smooth function. Assume that $ {\rm Vol}_g(\Omega) $ and the Euclidean diameter of $\Omega$ sufficiently small. In addition, we assume that the first eigenvalue of Laplace-Beltrami operator $ -\Delta_{g} $ on $ \Omega $ with Dirichlet condition satisfies $ \lambda_{1} \rightarrow \infty $ as $ \Omega $ shrinks. If $ R_{g} < 0 $ within the small enough closed domain $ \bar{\Omega} $, then the Dirichlet problem (\ref{local:eqn1}) has a real, positive, smooth solution $ u \in \calC^{\infty}(\Omega) \cap H_{0}^{1}(\Omega, g) \cap \calC^{0}(\bar{\Omega}) $.
\end{proposition}
\begin{proof}
As in \cite[\S2]{XU6}, first we show that (\ref{local:eqn8}) has a solution by applying Theorem \ref{local:thm1} again, provided that the domain is small enough. Without loss of generality, we may assume that $ \Omega $ is some geodesic normal ball of radius $ r $ centered at some point $ P $. As mentioned in Remark \ref{local:re1}, $ K < K_{0} $ in (\ref{local:eqn5}) is equivalent to the inequality
\begin{equation}\label{local:eqn9}
J_{1, \beta} : = \frac{\int_{\Omega} \sqrt{\det(g)} g^{ij} \partial_{i} u \partial_{j} u dx + \frac{1}{a} \int_{\Omega}  \sqrt{\det(g)} \left(R_{g} + \beta \right) u^{2} dx}{\left( \int_{\Omega} \sqrt{\det(g)} f u^{p} dx \right)^{\frac{2}{p}}} < \left(\max_{\bar{\Omega}}(f) \right)^{\frac{n-2}{n}} T
\end{equation}
for some appropriate choice of the test function $ u $. It is straightforward to check the equivalent between (\ref{local:eqn5}) and (\ref{local:eqn9}) by formulas in (\ref{local:eqn4}) and a standard argument for critical point of the functional $ J(u) $ in (\ref{local:eqn3}) with
\begin{equation*}
a_{ij} = a \sqrt{\det(g)} g^{ij}, b(x) = f(x) \sqrt{\det(g)}, F(x, u) = \frac{1}{2} \left( R_{g} + \beta \right) u^{2}.
\end{equation*}
We have shown this in \cite[Prop.~2.2]{XU6}, \cite[Prop.3.3]{XU3} for constant functions $ f $, only a very minor change is needed. It suffices to show that
\begin{equation}\label{local:eqn10}
J_{2, \beta, \Omega} : = \frac{\int_{\Omega} \sqrt{\det(g)} g^{ij} \partial_{i} u_{\beta} \partial_{j} u_{\beta} dx + \frac{1}{a} \int_{\Omega}  \sqrt{\det(g)} \left(R_{g} + \beta \right) u_{\beta}^{2} dx}{\left( \int_{\Omega} \sqrt{\det(g)} u_{\beta}^{p} dx \right)^{\frac{2}{p}}} < \left( \frac{\min_{\bar{\Omega}}(f)}{\max_{\bar{\Omega}}(f)} \right)^{\frac{n-2}{n}} T
\end{equation}
for some good choice of the test function $ u_{\beta} $. We showed in Appendix A of \cite{XU3} that for every $ \beta < 0 $, 
\begin{equation}\label{local:eqn11}
J_{2, \beta, \Omega} < T
\end{equation}
with the test function
\begin{equation*}
u_{\beta, \epsilon, \Omega} =  \frac{\varphi_{\Omega}(x)}{\left( \epsilon + \lvert x \rvert^{2} \right)^{\frac{n-2}{2}}}.
\end{equation*}
When $ n \geqslant 4 $, we choose $ \varphi_{\Omega} $ to be a radial cut-off function which is equal to $ 1 $ in a neighborhood of $ P $ and is equal to zero at the boundary. When $ n = 3 $, we apply a different function $ \varphi_{\Omega}(x) = \cos \left( \frac{\pi \lvert x \rvert}{2r} \right) $. We point out that the choice of $ \epsilon $ depends on the constant $ \beta $ and the size of the domain $ \Omega $. The smaller the $ \lvert \beta \rvert $ and/or the size of the domain $ \Omega $, the smaller the gap between $ J_{2, \beta, \Omega} $ and $ T $.

Next we show that for any small enough $ \Omega $ in the sense of volume and radial smallness, the functional $ J_{2, \beta, \Omega} $ satisfies
\begin{equation}\label{local:eqn12}
J_{2, \beta, \Omega} < J_{2, \beta_{0}, \Omega}, \forall \beta \in (\beta_{0}, 0), J_{2, \beta_{0}, \Omega} - J_{2, \beta, \Omega} > \delta > 0, \forall \beta \in (\beta_{0}, 0).
\end{equation}
Here $ \delta $ is a fixed positive constant, independent of the size of the domain and the choice of $ \beta $ and $ \epsilon $ in the old test function. To get this, we need a new test function. The details below are exactly the same as in \cite[Prop.~2.3]{XU6}. We consider the function $ v_{\beta, \Omega} $ satisfies
\begin{equation}\label{local:eqn13}
-a\Delta_{g} v_{\beta, \Omega} = -2R_{g} u_{\epsilon, \beta, \Omega} \; {\rm in} \; \Omega, v = 0 \; {\rm on} \; \partial \Omega.
\end{equation}
Note that $ R_{g} < 0 $ in $ \Omega $ hence $ v_{\beta, \Omega} > 0 $ by maximal principle. Denote
\begin{equation}\label{local:eqn14}
\begin{split}
\Gamma_{1} & : =  \frac{\int_{\Omega} (u_{\beta, \epsilon, \Omega} + v_{\beta, \Omega})^{p} \dvol}{\int_{\Omega} u_{\beta, \epsilon, \Omega}^{p} \dvol}; \\
\Gamma_{2} & : = \frac{a \int_{\Omega} \nabla_{g} (u_{\beta, \epsilon, \Omega} + v_{\beta, \Omega}) \cdot \nabla_{g}(u_{\beta, \epsilon, \Omega} + v_{\beta, \Omega}) \dvol + \int_{\Omega} \left( R_{g} + \beta \right) (u_{\beta, \epsilon, \Omega} + v_{\beta, \Omega})^{2} \dvol}{a \int_{\Omega} \nabla_{g} u_{\beta, \epsilon, \Omega} \cdot \nabla_{g}u_{\beta, \epsilon, \Omega} \dvol + \int_{\Omega} \left( R_{g} + \beta_{0} \right) u_{\beta, \epsilon, \Omega}^{2} \dvol}.
\end{split}
\end{equation}
By the same argument in \cite[Prop.~2.3]{XU6}, we showed that
\begin{equation*}
\Gamma_{2} \leqslant 1 + \frac{\left( 4 \inf_{\Omega} \lvert R_{g} \rvert+ \left( \beta - \beta_{0} \right) \right) \lambda_{1}^{-1}}{a  + \lambda_{1}^{-1} \left( - \sup_{\Omega} \lvert R_{g} \rvert + \beta_{0} \right)}.
\end{equation*} 
Here $ \lambda_{1} $ is the first eigenvalue of the Laplace-Beltrami operator and thus increases to positive infinity when the size of $ \Omega $ shrinks. We we consider $ \Gamma_{1} $ in a smaller geodesic ball, say the ball of radius $ \xi r $ centered at $ P $, $ \Gamma_{1} $ is given by the function $ u_{\beta, \epsilon, \Omega}(\xi x) $. Back to the PDE, it is the same as dealing with the PDE (\ref{local:eqn13}) with the metric $ \xi g $. The PDE (\ref{local:eqn13}) is scaling invariant, and thus the new solution is of the form $ v_{\beta, \Omega}(\xi x) $. Changing $ \epsilon $ will not affect the evaluation of $ \Gamma_{1} $ as both numerator and denominator changes at the same rate simultaneously. It follows that $ \Gamma_{1} $, which is larger than $ 1 $, has a uniform lower bound for all geodesic normal balls $ \Omega $ with radius $ \xi r $ centered at $ P $, $ \xi < 1 $, and all $ \epsilon > 0 $, i.e.
\begin{equation}\label{local:eqn15}
\Gamma_{1} \geqslant 1 + 2B
\end{equation}
for some constant $ \delta > 0 $. We then make $ \Omega $ small enough, which follows that
\begin{equation}\label{local:eqn16}
\Gamma_{2} \leqslant 1 + B
\end{equation}
due to the expression of $ \Gamma_{2} $. As we pointed out, we can apply $ u_{\beta, \epsilon, \Omega} $ into $ J_{2, \beta_{0}} $. The only difference is the choice of $ \epsilon $, which is smaller, thus $ T - J_{2, \beta_{0}, \Omega} $ with the test function $ u_{\beta, \epsilon, \Omega} $ is smaller but still positive.
Choosing the test function $ u_{\beta} = u_{\beta, \epsilon, \Omega} + v_{\beta, \Omega} $ for any $ J_{2, \beta, \Omega} $ with some $ \beta \in (\beta_{0}, 0) $, and the test function $ u_{\beta_{0}} = u_{\beta, \epsilon, \Omega} $ for $ J_{2, \beta_{0}, \Omega} $, we have shown that
\begin{equation}\label{local:eqn17}
\frac{J_{2, \beta, \Omega}}{J_{2, \beta_{0}, \Omega}} \leqslant \frac{1 + B}{1 + 2B} \leqslant 1 - \delta
\end{equation}
for some fixed $ \delta $. Note that (\ref{local:eqn17}) holds for all $ \beta \in (\beta_{0}, 0) $ and all geodesic normal balls $ \Omega $ with radius $ \xi r $ centered at $ P $, $ \xi < 1 $. We can make $ \Omega $ even smaller so that
\begin{equation*}
\left( \frac{\min_{\bar{\Omega}}(f)}{\max_{\bar{\Omega}}(f)} \right)^{\frac{n-2}{n}} > 1 - \frac{\delta}{2}.
\end{equation*}
It then follows that (\ref{local:eqn10}) holds for small enough $ \Omega $, due to (\ref{local:eqn17}). Not only that, we have also shown that there exists a constant $ \delta_{0} $ such that
\begin{align*}
J_{1, \beta} & \leqslant J_{2, \beta, \Omega} \cdot \left( \min_{\bar{\Omega}}(f) \right)^{\frac{2 - n}{n}} < (1 - \delta ) T \left( \min_{\bar{\Omega}}(f) \right)^{\frac{2 - n}{n}} \\
& = (1 - \delta ) \left( \frac{\max_{\bar{\Omega}}(f)}{\min_{\bar{\Omega}}(f)} \right)^{\frac{n - 2}{n}} \cdot \left(\max_{\bar{\Omega}}(f) \right)^{\frac{2 - n}{n}} T \\
& \leqslant \frac{1 - \delta}{1 - \frac{\delta}{2}} \cdot \left(\max_{\bar{\Omega}}(f) \right)^{\frac{2 - n}{n}} T.
\end{align*}
Due to the same argument in \cite[Prop.~2.2]{XU6} and \cite[Prop.3.3]{XU3}, we conclude that
\begin{equation}\label{local:eqn18}
K_{0} - \frac{1}{n} \left( a J_{1, \beta} \right)^{\frac{n}{2}} \geqslant \delta_{0} > 0
\end{equation}
for some fixed constant $ \delta_{0} $, provided that $ \Omega $ is small enough. The validity of (\ref{local:eqn10}) implies that all hypotheses in Theorem \ref{local:thm1} hold, hence the perturbed local Yamabe equation (\ref{local:eqn8}) has a solution, for all $ \beta < 0 $.

We now have a sequence of positive, smooth solutions $ \lbrace u_{\beta, *} \rbrace $ that solve (\ref{local:eqn8}) for each $ \beta < 0 $. To take the limit $ \beta \rightarrow 0^{-} $, we need to estimate $ \lVert u_{\beta, *} \rVert_{\calL^{t}(\Omega, g)} $ for some $ t > p  = \frac{2n}{n - 2} $. We consider the norm within a region $ \beta \in (\beta_{0}, 0 $. The choice of $ \beta_{0} $ is quite flexible. According to Wang's result \cite[Thm.~1.1]{WANG}, we know that the solutions of (\ref{local:eqn8}) satisfies
\begin{equation*}
J(u_{\beta, *}) \leqslant \frac{1}{n} \left( a J_{1, \beta} \right)^{\frac{n}{2}} \leqslant K_{0} - \delta_{0}, \forall \beta \in (\beta_{0}, 0).
\end{equation*}
Pairing (\ref{local:eqn8}) with the solution $ u_{\beta, *} $ on both sides, we have
\begin{equation}\label{local:eqn19}
a \lVert \nabla_{g} u_{\beta, *} \rVert_{\calL^{2}(\Omega, g)}^{2} + \int_{\Omega} \left( R_{g} + \beta \right) u_{\beta, *}^{2} \dvol = \int_{\Omega} f u_{\beta, *}^{p} \dvol.
\end{equation}
Recall that $ p = \frac{2n}{n - 2} $, we read $ J(u_{\beta, *}) \leqslant K_{0} - \delta_{0} $ with the equality (\ref{local:eqn19}) as
\begin{align*}
J(u_{\beta, *}) \leqslant K_{0} - \delta_{0} & \Leftrightarrow \frac{a}{2} \lVert \nabla_{g} u_{\beta, *} \rVert_{\calL^{2}(\Omega, g)}^{2} - \frac{1}{p}  \int_{\Omega} f u_{\beta, *}^{p} \dvol + \frac{1}{2} \int_{\Omega} \left( R_{g} + \beta \right) u_{\beta, *}^{2} \dvol \leqslant K_{0} - \delta_{0} \\
& \Leftrightarrow  \frac{a}{2} \lVert \nabla_{g} u_{\beta, *} \rVert_{\calL^{2}(\Omega, g)}^{2} - \frac{n - 2}{2n} \left( a \lVert \nabla_{g} u_{\beta, *} \rVert_{\calL^{2}(\Omega, g)}^{2} + \int_{\Omega} \left( R_{g} + \beta \right) u_{\beta, *}^{2} \dvol\right) \\
& \qquad + \frac{1}{2} \int_{\Omega} \left( R_{g} + \beta \right) u_{\beta, *}^{2} \dvol \leqslant K_{0} - \delta_{0} \\
& \Leftrightarrow \frac{a}{n} \lVert \nabla_{g} u_{\beta, *} \rVert_{\calL^{2}(\Omega, g)}^{2} + \frac{1}{n}  \int_{\Omega} \left( R_{g} + \beta \right) u_{\beta, *}^{2} \dvol \leqslant K_{0} - \delta_{0}.
\end{align*}
Note that
\begin{equation*}
K_{0} - \delta_{0} = \frac{K_{0} - \delta_{0}}{K_{0}} \cdot K_{0} = \frac{K_{0} - \delta_{0}}{K_{0}} \cdot \frac{1}{n} \left( \max_{\bar{\Omega}}(f) \right)^{\frac{2 - n}{2}} a^{\frac{n}{2}} T^{\frac{n}{2}}.
\end{equation*}
It follows from previous two calculations that
\begin{equation}\label{local:eqn20}
a \lVert \nabla_{g} u_{\beta, *} \rVert_{\calL^{2}(\Omega, g)}^{2} +  \int_{\Omega} \left( R_{g} + \beta \right) u_{\beta, *}^{2} \dvol \leqslant \frac{K_{0} - \delta_{0}}{K_{0}} \cdot  \left( \max_{\bar{\Omega}}(f) \right)^{\frac{2 - n}{2}} a^{\frac{n}{2}}.
\end{equation}
Put (\ref{local:eqn20}) back into (\ref{local:eqn19}), we conclude that
\begin{equation}\label{local:eqn21}
\int_{\Omega} f u_{\beta, *}^{p} \dvol \leqslant \frac{K_{0} - \delta_{0}}{K_{0}} \cdot  \left( \max_{\bar{\Omega}}(f) \right)^{\frac{2 - n}{2}} a^{\frac{n}{2}}, \forall \beta \in (\beta_{0}, 0).
\end{equation}
Let $ \theta > 0 $ be some constant that will be determined later. Define
\begin{equation*}
w_{\beta} : = u_{\beta, *}^{1 + \theta}
\end{equation*}
and pair the perturbed local Yamabe equation (\ref{local:eqn8}) with $ u_{\beta, *}^{1 + 2\theta} $ on both sides, we have
\begin{align*}
& \int_{\Omega} a \nabla_{g} u_{\beta, *} \cdot \nabla_{g} \left(u_{\beta, *}^{1 + 2\theta} \right) \dvol + \int_{\Omega} \left(R_{g} + \beta \right) u_{\beta, *}^{2 + 2\theta} \dvol = \int_{\Omega} f u_{\beta, *}^{p + 2\theta} \dvol; \\
\Rightarrow & \frac{1 + 2\theta}{(1 + \theta )^{2}} \int_{\Omega} a \lvert \nabla_{g} w_{\beta} \rvert^{2} \dvol =  \int_{\Omega} f w_{\beta}^{2} u_{\beta, *}^{p-2} \dvol - \int_{\Omega} \left(R_{g} + \beta \right) w_{\beta}^{2} \dvol \\
& \qquad \leqslant \left( \max_{\bar{\Omega}} (f) \right)^{\frac{n - 2}{n}} \int_{\Omega} w_{\beta}^{2} f^{\frac{p-2}{p}} u_{\beta, *}^{p-2} \dvol - \int_{\Omega} \left(R_{g} + \beta \right) w_{\beta}^{2} \dvol.
\end{align*}
The last inequality holds since $ f, w_{\beta}, u_{\beta, *} $ are all positive functions. By the sharp Sobolev inequality on closed manifolds \cite[Thm.~2.3, Thm.~3.3]{PL}, it follows that for any $ \alpha > 0 $,  
\begin{align*}
\lVert w_{\beta} \rVert_{\calL^{p}(\Omega, g)}^{2} & \leqslant (1 + \alpha) \frac{1}{T} \lVert \nabla_{g} w_{\beta} \rVert_{\calL^{2}(\Omega, g)}^{2} + C_{\alpha}' \lVert w_{\beta} \rVert_{\calL^{2}(\Omega, g)}^{2} \\
& = (1 + \alpha) \frac{1}{aT} \cdot \frac{(1 + \theta )^{2}}{1 + 2\theta} \left( \frac{1 + 2\theta}{(1 + \theta )^{2}} \int_{\Omega} a \lvert \nabla_{g} w_{\beta} \rvert^{2} \dvol \right) + C_{\alpha}' \lVert w_{\beta} \rVert_{\calL^{2}(\Omega, g)}^{2} \\
& \leqslant (1 + \alpha) \frac{\left( \max_{\bar{\Omega}} (f) \right)^{\frac{n - 2}{n}}}{aT} \cdot \frac{(1 + \theta )^{2}}{1 + 2\theta} \int_{\Omega} w_{\beta}^{2} f^{\frac{p-2}{p}} u_{\beta, *}^{p - 2} \dvol + C_{\alpha}  \lVert w_{\beta} \rVert_{\calL^{2}(\Omega, g)}^{2} \\
& \leqslant (1 + \alpha) \frac{\left( \max_{\bar{\Omega}} (f) \right)^{\frac{n - 2}{n}}}{aT} \cdot \frac{(1 + \theta )^{2}}{1 + 2\theta} \lVert w_{\beta} \rVert_{\calL^{p}(\Omega, g)}^{2} \left( \int_{\Omega} f u_{\beta, *}^{p} \dvol \right)^{p - 2} + C_{\alpha}  \lVert w_{\beta} \rVert_{\calL^{2}(\Omega, g)}^{2} \\
& \leqslant (1 + \alpha) \frac{\left( \max_{\bar{\Omega}} (f) \right)^{\frac{n - 2}{n}}}{aT} \cdot \frac{(1 + \theta )^{2}}{1 + 2\theta} \lVert w_{\beta} \rVert_{\calL^{p}(\Omega, g)}^{2} \cdot \left( \frac{K_{0} - \delta_{0}}{K_{0}} \left( \max_{\bar{\Omega}} (f) \right)^{\frac{2 - n}{2}}a^{\frac{n}{2}} T^{\frac{n}{2}} \right)^{\frac{2}{n}} \\
& \qquad + C_{\alpha}  \lVert w_{\beta} \rVert_{\calL^{2}(\Omega, g)}^{2} \\
& = (1 + \alpha) \cdot \frac{(1 + \theta )^{2}}{1 + 2\theta} \cdot \left( \frac{K_{0} - \delta_{0}}{K_{0}} \right)^{\frac{2}{n}} \lVert w_{\beta} \rVert_{\calL^{p}(\Omega, g)}^{2} + C_{\alpha}  \lVert w_{\beta} \rVert_{\calL^{2}(\Omega, g)}^{2}.
\end{align*}
It follows that we can choose appropriate positive constants $ \alpha, \theta $ so that
\begin{equation*}
(1 + \alpha) \cdot \frac{(1 + \theta )^{2}}{1 + 2\theta} \cdot \left( \frac{K_{0} - \delta_{0}}{K_{0}} \right)^{\frac{2}{n}} < 1, \forall \beta \in (\beta_{0}, 0).
\end{equation*}
Hence we have
\begin{equation}\label{local:eqn22}
\lVert w_{\beta} \rVert_{\calL^{p}(\Omega, g)}^{2} = \lVert u_{\beta, *} \rVert_{\calL^{p + p \cdot \theta}(\Omega, g)}^{1 + \theta} < C, \forall \beta \in (\beta_{0}, 0).
\end{equation}
Note that $ p + p \cdot \theta > 2 + \delta_{1} $ for some positive constant $ \delta_{1} $. The rest of the argument is exactly the same as in \cite[Prop.~2.4]{XU6}, which applies bootstrapping methods, Arzela-Ascoli theorem, etc., to take the limit $ \beta \rightarrow 0^{-} $ in the classical sense, and then to get a positive, smooth solution of (\ref{local:eqn1}). 
\end{proof}

When the manifold is locally conformally flat, we can get the existence result in a different, simpler way. Consider a manifold $ (\bar{M}, g) $, with or without boundary. For the conformal change $ \tilde{g} = \varphi^{p-2} g $, we have
\begin{equation}\label{local:eqn23}
\left( -a\Delta_{\tilde{g}} + R_{\tilde{g}} \right) u = \varphi^{\frac{n - 2}{n + 2}} \left( -a\Delta_{g} + R_{g} \right) \left( \varphi u \right), u \in \calC^{\infty}(\bar{M}).
\end{equation}
This is the conformal invariance of the conformal Laplacian. The next result relies on the existence of the solutions of the nonlinear eigenvalue problem $ -\Delta_{e} u = f u^{p-1} $ in some open subset $ \Omega \subset \R^{n} $, with Dirichlet boundary condition. When $ f $ is a constant function, we refer to Bahri and Coron \cite{BC}; when $ f $ is a general positive function, we refer to Clapp, Faya and Pistoia \cite{CFP}.
\begin{proposition}\label{local:prop3}\cite[Prop.~2.5]{XU6}
Let $ (\Omega, g) $ be a Riemannian domain in $\R^n$, $ n \geqslant 3 $, with $C^{\infty} $ boundary. Let the metric $ g $ be locally conformally flat on some open subset $ \Omega' \supset \bar{\Omega} $. For any point $ \rho \in \Omega $ and any positive constant $ \epsilon $, denote the region $ \Omega_{\epsilon} $ to be
\begin{equation*}
\Omega_{\epsilon} = \lbrace x \in \Omega | \lvert x - \rho \rvert > \epsilon \rbrace.
\end{equation*}
Assume that $ Q \in \calC^{2}(\bar{\Omega}) $, $ \min_{x \in \bar{\Omega}} Q(x) > 0 $ and $ \nabla Q(\rho) \neq 0 $. Then there exists some $ \epsilon_{0} $ such that for every $ \epsilon \in (0, \epsilon_{0}) $ the Dirichlet problem
\begin{equation}\label{local:eqn24}
-a\Delta_{g}u + R_{g} u = Qu^{p-1} \; {\rm in} \; \Omega_{\epsilon}, u = 0 \; {\rm on} \; \partial \Omega_{\epsilon}
\end{equation}
has a real, positive, smooth solution $ u \in \calC^{\infty}(\Omega_{\epsilon}) \cap H_{0}^{1}(\Omega_{\epsilon}, g) \cap \calC^{0}(\bar{\Omega_{\epsilon}}) $.
\end{proposition}
\begin{remark}\label{local:re2}
The results in Proposition \ref{local:prop2} and \ref{local:prop3} are all we need for this article. We observe that when the manifold is locally conformally flat, the local Yamabe problem can only have nontrivial solution within a topologically nontrivial set, like annulus. We would like to point out that Wang's result \cite[Thm.~1.1]{WANG} can be used to get more interesting local results, especially at critical exponent. For super-critical exponent, people apply Hopf fiberation to get interesting local results, we refer to \cite{CFP}.
\end{remark}

\section{Necessary and Sufficient Conditions for Prescribed Scalar Curvature Problem with Zero First Eigenvalue}
In this section, we show two main results. For closed manifolds $ (M, g) $, $ n = \dim M \geqslant 3 $ with zero first eigenvalue of the conformal Laplacian, the necessary and sufficient condition for a given function $ S $ to be realized as a prescribed scalar curvature function for some pointwise conformal metric $ \tilde{g} \in [g] $ is either $ S \equiv 0 $ or $ S $ changes sign and $ \int_{M} S \dvol < 0 $. This problem is equivalent to the existence of some positive, smooth solution of the PDE
\begin{equation}\label{zero:eqn1}
-a\Delta_{g} u = S u^{p-1} \; {\rm in} \; M
\end{equation}
since we may assume the initial metric $ g $ is scalar-flat due to the Yamabe problem, see \cite{XU3}. We always start with the scalar-flat metric $ g $ throughout this section. The closed Riemann surface case was settled by Kazdan and Warner \cite{KW2}. The closed manifolds with dimensions $ 3 $ and $ 4 $ were settled by Escobar and Schoen \cite{ESS}. We extend the same necessary and sufficient condition to all closed manifolds with dimensions at least $ 3 $, provided that the first eigenvalue of the conformal Laplacian is zero.

Next we consider the analogy on compact manifolds $ (\bar{M}, g) $ with non-empty smooth boundary, $ n = \dim \bar{M} \geqslant 3 $, again with zero first eigenvalue of the conformal Laplacian and Robin boundary condition. We show that the necessary and sufficient condition for a given function $ S $ to be realized as a prescribed scalar curvature function for some Yamabe metric $ \tilde{g} \in [g] $ with minimal boundary, i.e. the mean curvature $ h_{\tilde{g}} = 0 $ is either $ S \equiv 0 $ or $ S $ changes sign and $ \int_{M} S \dvol < 0 $. Based on our best understanding, this ultimate result is given for the first time. Again we always start with the scalar-flat metric $ g $ with minimal boundary in this case. This problem is equivalent to the existence of some positive, smooth solution of the PDE
\begin{equation}\label{zero:eqn2}
-a\Delta_{g} u = Su^{p-1} \; {\rm in} \; M, \frac{\partial u}{\partial \nu} = 0 \; {\rm on} \; \partial M
\end{equation}
since we may assume the initial metric $ g $ is both scalar-flat and mean-flat due to the Escobar problem, see \cite{XU4}.

The key step for both cases is to construct both lower and upper solutions of the Yamabe equations (\ref{zero:eqn1}) or (\ref{zero:eqn2}). If $ S \equiv 0 $, then the problem is trivial. We are interested in the nontrivial case from now on. In order to construct the upper solution, we need to observe the following relation between (\ref{zero:eqn1}) or (\ref{zero:eqn2}) and a new PDE. This is essentially due to Kazdan and Warner \cite{KW}.
\begin{lemma}\label{zero:lemma1}
Let $ (M, g) $ be a closed manifold, $ n \geqslant 3 $. Let $ S \in \calC^{\infty}(M) $ be a given function. Then there exists some positive function $ u \in \calC^{\infty}(M) $ satisfying
\begin{equation}\label{zero:eqn3}
-a\Delta_{g} u \geqslant S u^{p-1} \; {\rm in} \; M
\end{equation}
if and only if there exists some positive function $ v \in \calC^{\infty}(M) $ satisfying
\begin{equation}\label{zero:eqn4}
-a\Delta_{g} v + \frac{(p - 1)a}{p - 2} \cdot \frac{\lvert \nabla_{g} v \rvert^{2}}{v} \leqslant (2 - p)S.
\end{equation}
Moreover, the equality in (\ref{zero:eqn3}) holds if and only if the equality in (\ref{zero:eqn4}) holds.
\end{lemma}
\begin{proof} Assume that there is a positive function $ u \in \calC^{\infty}(M) $ that satisfies (\ref{zero:eqn3}). Define
\begin{equation*}
v = u^{2 - p}.
\end{equation*}
Note that $ 2 - p = -\frac{4}{n - 2} < 0 $ since $ n \geqslant 3 $ by hypothesis. We compute that
\begin{equation*}
\nabla v = (2 - p) u^{1 - p} \nabla u \Leftrightarrow \nabla u = u^{p-1} (2 - p)^{-1} \nabla v,
\end{equation*}
and 
\begin{equation*}
\Delta_{g} v = (2 - p) u^{1 - p} \Delta_{g} u + (2 - p)(1 - p) u^{-p} \lvert \nabla_{g} u \rvert^{2}.
\end{equation*}
Using the relation between $ \nabla u $ and $ \nabla v $, $ v = u^{2-p} $ and the inequality (\ref{zero:eqn3}), we have
\begin{align*}
a\Delta_{g} v & = (2 - p) u^{1 - p} \left( a\Delta_{g} u \right) +a (2 - p) (1 - p) u^{-p} \lvert \nabla_{g} u \rvert^{2} \\
& \geqslant (p - 2) u^{1-p} S u^{p-1} + a(2 - p) (1 - p) (2 - p)^{-2} u^{2p - 2} u^{-p} \lvert \nabla_{g} v \rvert^{2} \\
& = (p - 2) S + \frac{a(p - 1)}{p - 2} u^{p -2} \lvert \nabla_{g} v \rvert^{2} = (p - 2) S + \frac{a(p - 1)}{p - 2} \frac{ \lvert \nabla_{g} v \rvert^{2}}{v}.
\end{align*}
Shifting $ (p -2) S $ to the left side and $ a\Delta_{g} v $ to the right side, we get the inequality (\ref{zero:eqn4}). For the other direction, we start with some positive function $ v \in \calC^{\infty}(M) $ and define $ u = v^{\frac{1}{2 - p}} $. The following argument is quite similar, we omit it.

The equality part is also straightforward, we just need to change inequalities above into equalities.
\end{proof}
We now construct a good candidate of the upper solution of (\ref{zero:eqn1}) on closed manifolds $ (M, g) $.
\begin{proposition}\label{zero:prop1} 
Let $ (M, g) $ be a closed manifold, $ n = \dim M \geqslant 3 $. Let $ S \not\equiv 0 $ be a given smooth such that
\begin{equation}\label{zero:eqn5}
\int_{M} S \dvol < 0.
\end{equation}
If $ \eta_{1} = 0 $, there exists a positive function $ u \in \calC^{\infty}(M) $ such that
\begin{equation}\label{zero:eqn6}
-a\Delta_{g} u \geqslant S u^{p-1} \; {\rm in} \; M.
\end{equation}
\end{proposition}
\begin{proof} By Lemma \ref{zero:lemma1}, it suffice to show that there exists some positive function $ v \in \calC^{\infty}(M) $ such that (\ref{zero:eqn4}) holds. Note that (\ref{zero:eqn6}) implies that $ \int_{M} (2 - p) S \dvol > 0 $ since $ 2 - p < 0 $. Define
\begin{equation*}
\gamma : = \frac{2 - p}{\text{Vol}_{g}(M)} \int_{M} S \dvol.
\end{equation*}
We consider the PDE
\begin{equation}\label{zero:eqn7}
-a\Delta_{g} v_{0} = (2 - p ) S - \gamma \; {\rm in} \; M.
\end{equation}
Due to the definition of $ \gamma $, we observe that the average of the function $ (2 - p) S - \gamma $ is zero. By standard elliptic theory, there exists a smooth solution $ v_{0} $ of (\ref{zero:eqn7}). Fix this $ v_{0} $. Take $ C $ large enough so that
\begin{equation*}
v : = v_{0} + C > 0 \; {\rm on} \; M.
\end{equation*}
Clearly $ v $ solves (\ref{zero:eqn7}). We enlarge the constant $ C $, if necessary so that
\begin{equation*}
\frac{a(p - 1)}{ p - 2} \frac{\lvert \nabla_{g} v \rvert^{2}}{v} = \frac{a(p - 1)}{ p - 2} \frac{\lvert \nabla_{g} v_{0} \rvert^{2}}{v_{0} + C} < \gamma.
\end{equation*}
Fix the constant $ C $. It follows for this choice of $ v $, we have
\begin{equation*}
-a\Delta_{v} + \frac{(p - 1)a}{p - 2} \cdot \frac{\lvert \nabla_{g} v \rvert^{2}}{v} = ( 2- p) S - \gamma + \frac{(p - 1)a}{p - 2} \cdot \frac{\lvert \nabla_{g} v \rvert^{2}}{v} \leqslant (2 - p) S - \gamma + \gamma \leqslant (2 - p) S.
\end{equation*}
Therefore (\ref{zero:eqn4}) holds.
\end{proof}
\begin{remark}\label{zero:re1}
Note that the assumption $ \int_{M} S \dvol < 0 $ is essential here since otherwise we may not be able to compensate the positive term $  \frac{(p - 1)a}{p - 2} \cdot \frac{\lvert \nabla_{g} v \rvert^{2}}{v} $ pointwise.

Note also that the construction of the upper solution does not require $ S $ to change sign. This condition is necessary when we apply our local analysis in \S3 to construct a lower solution.

Lastly, we need to say that the validity of (\ref{zero:eqn6}) may not be good enough. In order to apply the monotone iteration scheme stated in \S2, the function $ u $ must also be larger than the lower solution we construct. We will handle this later.
\end{remark}
As we mentioned in Remark \ref{zero:re1}, we have to consider the largeness of $ u $. However, it is in general very hard to get uniform largeness on the whole manifold $ M $ unless we are dealing with the negative first eigenvalue case. We therefore introduce the next result as the preparation for the construction of the upper solution.
\begin{corollary}\label{zero:cor1}
Let $ (M, g) $ be a closed manifold, $ n = \dim M \geqslant 3 $. Let $ S \not\equiv 0 $ be a given smooth such that
\begin{equation}\label{zero:eqn8}
\int_{M} S \dvol < 0.
\end{equation}
If $ \eta_{1} = 0 $, there exists a positive function $ u \in \calC^{\infty}(M) $ and a small enough constant $ \gamma_{0} > 0 $ such that
\begin{equation}\label{zero:eqn9}
-a\Delta_{g} u \geqslant \left( S + \gamma_{0} \right) u^{p-1} \; {\rm in} \; M.
\end{equation}
\end{corollary}
\begin{proof} Since $ \int_{M} S \dvol< 0 $, we can always find some constant $ \gamma_{0} > 0 $ such that
\begin{equation*}
\int_{M} \left( S + \gamma_{0} \right) \dvol < 0.
\end{equation*}
We then repeat the argument in Proposition \ref{zero:prop1} for the new function $ S + \gamma_{0} $.
\end{proof}
We now construct the lower solution. There is still an obstacle to construct a lower solution. If the manifold is locally conformally flat, we apply Proposition \ref{local:prop3}) directly since locally the metric is conformal to the Euclidean metric. Precisely speaking, the metric $ g = \phi^{p-2} g_{e} $ for some positive function $ \phi $. By (\ref{local:eqn23}), $ u $ solves (\ref{local:eqn24}) is equivalent to 
\begin{equation*}
-a\Delta_{e} (\phi u) = Q (\phi u)^{p-1} \; {\rm in} \; \Omega_{\epsilon}, \phi u = 0 \; {\rm on} \; \partial \Omega_{\epsilon}.
\end{equation*}
Extend the function $ \phi u $ by zero serves as a good candidate of the lower solution. But if the manifold is not locally conformally flat, we will use the result of Proposition \ref{local:prop2}, in which we require $ R_{g} < 0 $ within the small domain $ \Omega $. Not only that, we require the region on which $ R_{g} < 0 $ overlaps the region on which $ S > 0 $. The next result, which was first given in \cite[Thm.~4.6]{XU3}, resolves this issue.
\begin{proposition}\label{zero:prop2}
Let $ (M, g) $ be a closed manifold with $ n = \dim M \geqslant 3 $. Let $ P $ be any fixed point of $ M $. Then there exists some smooth function $ R_{0} \in \calC^{\infty}(M) $, which is negative at $ P $, such that $ R_{0} $ is realized as the scalar curvature function with respect to some conformal change of the original metric $ g $.
\end{proposition}
\begin{proof}
See \cite[Thm.~4.6]{XU3}. Although \cite[Thm.~4.6]{XU3} is for the positive first eigenvalue case. The same argument applies here. The key observation is that even for the scalar-flat case, the size of the geodesic ball is also comparable with the size of the Euclidean ball with the same radius, provided that the radius is small enough.
\end{proof}
\medskip

We now construct the lower solution of the PDE (\ref{zero:eqn1}).
\begin{proposition}\label{zero:prop3}
Let $ (M, g) $ be a closed manifold, $ n = \dim M \geqslant 3 $. Let $ S \not\equiv 0 $ be a given smooth function on $ M $; in addition, $ S $ changes sign. If the metric $ g $ is scalar-flat, then there exists a nonnegative function $ u_{-} \in \calC^{0}(M) \cap H^{1}(M, g) $, not identically zero, such that
\begin{equation}\label{zero:eqn10}
-a\Delta_{g} u_{-} \leqslant S u_{-}^{p-1} \; {\rm in} \; M
\end{equation}
holds in the weak sense.
\end{proposition}
\begin{proof}
We classify the argument in terms of the vanishing of the Weyl tensor. Let $ O $ be the set on which $ S > 0 $. Since $ S $ changes sign, we may assume, without loss of generality, that there exist a point in $ O $ at which $ \nabla S \neq 0 $.

We assume first that there exists a point $ P \in O $ such that the Weyl tensor does not vanish, so is for some neighborhood of $ P $. By Proposition \ref{zero:prop2}, there exists a conformal metric $ \tilde{g} = v^{p-2} g $, $ v \in \calC^{\infty}(M) $ is positive, such that the scalar curvature $ R_{\tilde{g}} < 0 $ on $ P $ and hence in some neighborhood of $ P $. Due to the conformal invariance of the conformal Laplacian in (\ref{local:eqn23}), it suffices to consider the PDE
\begin{equation}\label{zero:eqn11}
-a\Delta_{\tilde{g}} u_{0} + R_{\tilde{g}} u_{0} = S u_{0}^{p-1} \; {\rm in} \; \Omega, u_{0} = 0 \; {\rm on} \; \partial \Omega.
\end{equation}
Here $ \Omega $ is any open Riemannian domain on which $ S > 0 $, the Weyl tensor does not vanish, and $ R_{\tilde{g}} < 0 $. Shrinking $ \Omega $ if necessary, we conclude from Proposition \ref{local:prop2} that (\ref{zero:eqn11}) has a positive solution $ u_{0} \in \calC^{\infty}(\Omega) \cap H_{0}^{1}(\Omega, \tilde{g}) \cap \calC^{0}(\bar{\Omega}) $. Define
\begin{equation}\label{zero:eqn12}
\tilde{u}_{-} : = \begin{cases} u_{0}, & \; {\rm in} \; \Omega \\ 0, & \; {\rm in} \; M \backslash \Omega \end{cases}.
\end{equation}
It is easy to check that $ \tilde{u}_{-} \geqslant 0 $, not identically zero, $ \tilde{u}_{-} \in H^{1}(M, g) \cap \calC^{0}(M) $. By the same reason as in, e.g. \cite[Thm.~4.4]{XU3}, we can check that $ \tilde{u}_{-} $ satisfies
\begin{equation*}
-a\Delta_{\tilde{g}} \tilde{u}_{-} + R_{\tilde{g}} \tilde{u}_{-} \leqslant S \tilde{u}_{-}^{p - 1} \: {\rm in} \; M
\end{equation*}
in the weak sense. Due to Aubin's result for conformal invariance of the Yamabe quotient \cite[\S5.8]{Aubin}, we observe that for any nonnegative test function $ \phi \in \calC^{0}(M) \cap H^{1}(M, g) $, we have
\begin{align*}
& \int_{M} a\nabla_{g} (v \tilde{u}_{-}) \cdot \nabla_{g} (v \phi) \dvol - \int_{M} S \left( v\tilde{u}_{-} \right)^{p-1} \left( v \phi \right) \dvol \\
& \qquad = \int_{M} a \nabla_{\tilde{g}} \tilde{u}_{-} \nabla_{\tilde{g}} \phi d\text{Vol}_{\tilde{g}} + \int_{M} R_{\tilde{g}} \tilde{u}_{-} \cdot \phi d\text{Vol}_{\tilde{g}} + \int_{M} S \tilde{u}_{-}^{p-1} \phi d\text{Vol}_{\tilde{g}} \leqslant 0.
\end{align*}
Here we use the assumption that $ \eta_{1} = 0 $ thus the model case is scalar-flat due to the Yamabe problem. Since $ v $ is a positive, smooth function, the map
\begin{equation*}
v \mapsto v \phi, H^{1}(M, g) \rightarrow H^{1}(M, g)
\end{equation*}
is an isomorphism. It follows that (\ref{zero:eqn10}) holds weakly for
\begin{equation}\label{zero:eqn13}
u_{-} : = v\tilde{u}_{-}.
\end{equation}
Clearly $ u_{-} \geqslant 0 $, not identically zero, and $ u_{-} \in \calC^{0}(M) \cap H^{1}(M, g) $.

If the Weyl tensor is identically zero within the domain $ O $, we apply Proposition \ref{local:prop3} to get a positive, smooth solution of (\ref{zero:eqn11}) within a different region $ \Omega $ given in Proposition \ref{local:prop3}, centered at some point $ P $ at which $ \nabla S \neq 0 $. The rest of the argument is exactly the same.
\end{proof}

We do not know whether the function $ u $ in (\ref{local:eqn6}) is larger than $ u_{-} $ pointwise. To resolve this issue, we apply the local gluing strategy in \cite[Lemma~3.2]{XU6} and the upper solution in Corollary \ref{zero:cor1} to construct an upper solution.
\begin{proposition}\label{zero:prop4}
Let $ (M, g) $ be a closed manifold, $ n = \dim M \geqslant 3 $. Let $ S \not\equiv 0 $ be a given smooth function on $ M $; in addition, $ S $ changes sign. If the metric $ g $ is scalar-flat, then there exists a positive function $ u_{+} \in \calC^{\infty}(M) $ such that
\begin{equation}\label{zero:eqn14}
-a\Delta_{g} u_{+} \geqslant S u_{+}^{p-1} \; {\rm in} \; M.
\end{equation}
Furthermore, $ 0 \leqslant u_{-} \leqslant u_{+} $ pointwise on $ M $.
\end{proposition}
\begin{proof}
By Corollary \ref{zero:cor1}, we see that there exists a positive, smooth function $ u_{1} $ satisfying (\ref{zero:eqn9}). By conformal invariance of the conformal Laplacian in the strong sense (\ref{local:eqn23}), the same conformal change $ \tilde{g} = v^{p-2} g $ given in Proposition \ref{zero:prop3} above implies that
\begin{equation}\label{zero:eqn15}
\tilde{u}_{1} : = v u_{1} \Rightarrow -a\Delta_{\tilde{g}} \tilde{u}_{1} + R_{\tilde{g}} \tilde{u}_{1} \geqslant \left( S + \gamma_{0} \right) \tilde{u}_{1}^{p-1} \; {\rm in} \; M.
\end{equation}
This inequality holds since $ v > 0 $ on $ M $ and $ \eta_{1} = 0 $. Gluing functions $ \tilde{u}_{1} $ in (\ref{zero:eqn15}) and $ u_{0} $ in (\ref{zero:eqn11}) by exactly the same argument in \cite[Lemma~3.2]{XU6}, we conclude that there exists a positive, smooth function $ \tilde{u}_{2} \in \Omega $ such that
\begin{equation}\label{zero:eqn16}
\begin{split}
& -a\Delta_{\tilde{g}} \tilde{u}_{2} + R_{\tilde{g}} \tilde{u}_{2} \geqslant S \tilde{u}_{2}^{p-1} \; {\rm in} \; \Omega; \\
& \tilde{u}_{2} \geqslant u_{0} \; {\rm in} \; \Omega, \tilde{u}_{2} = \tilde{u}_{1} \; {\rm in} \; \Omega_{0} : = \lbrace x \in \Omega : d(x, \partial \Omega) < \xi \; \text{for some $ \xi > 0 $} \rbrace.
\end{split}
\end{equation}
For the full detail the gluing procedure we refer to \cite[Thm.~4.4]{XU3}. We point out that we need the extra $ \gamma_{0} \tilde{u}_{1}^{p-1} $ term in (\ref{zero:eqn15}) to give us some room to compensate the negative terms on $ \partial \Omega $. Precisely speaking, the term $ \beta $ in formula \cite[(49)]{XU6} is given by $ \beta = \max_{\Omega} \gamma_{0} \tilde{u}_{1}^{p-1} $ here.
Define
\begin{equation}\label{zero:eqn17}
\tilde{u}_{+} : = \begin{cases} \tilde{u}_{2}, & \; {\rm in} \; \Omega \\ \tilde{u}_{1}, & \; {\rm in} \; M \backslash \bar{\Omega} \end{cases}.
\end{equation}
It is straightforward to check that $ \tilde{u}_{+} $ is a positive, smooth function on $ M $ and satisfies
\begin{equation}\label{zero:eqn18}
-a\Delta_{\tilde{g}} \tilde{u}_{+} + R_{\tilde{g}} \tilde{u}_{+} \geqslant S \tilde{u}_{+}^{p-1} \; {\rm in} \; M.
\end{equation}
In addition, $ \tilde{u}_{+} \geqslant \tilde{u}_{-} \geqslant 0 $ on $ M $. Define
\begin{equation}\label{zero:eqn19}
u_{+} : = v\tilde{u}_{+},
\end{equation}
The conformal invariance of the conformal Laplacian indicates that (\ref{zero:eqn14}) holds, provided that $ \eta_{1} = 0 $, or equivalently, $ g $ is scalar-flat. Compare the definitions of $ u_{-} $ and $ u_{+} $ in (\ref{zero:eqn13}) and (\ref{zero:eqn19}) respectively, we conclude that
\begin{equation*}
u_{+} \geqslant u_{-} \geqslant 0 \; \text{pointwise in $ M $}.
\end{equation*}
\end{proof}
\medskip

Now we can show the necessary and sufficient condition of prescribed scalar curvature problem within a pointwise conformal class of metrics $ [g] $ on the closed manifolds $ (M, g) $, $ n = \dim M \geqslant 3 $, provided that $ \eta_{1} = 0 $.
\begin{theorem}\label{zero:thm1}
Let $ (M, g)) $ be a closed manifold, $ n = \dim M \geqslant 3 $. Let $ S \in \calC^{\infty}(M) $ be a given function. Assume that $ \eta_{1} = 0 $. The function $ S $ can be realized as the prescribed scalar curvature of some conformal metric $ \tilde{g} \in [g] $ if and only if
\begin{enumerate}[(i).]
\item $ S \equiv 0 $ on $ M $;
\item $ S $ changes sign, and $ \int_{M} S \dvol < 0 $.
\end{enumerate}
\end{theorem}
\begin{proof} Since $ \eta_{1} = 0 $, we may assume that $ g $ is scalar-flat since otherwise we can arrange a conformal change to get this in advance. Then this problem is reduced to the existence of some positive solution $ u \in \calC^{\infty}(M) $ such that
\begin{equation}\label{zero:eqn20}
-a\Delta_{g} u = S u^{p-1} \; {\rm in} \; M.
\end{equation}
We consider the sufficient condition first. When $ S \equiv 0 $, the problem is trivial by taking $ u \equiv 1 $. When $ S $ satisfies (ii), we apply Proposition \ref{zero:prop3} and Proposition \ref{zero:prop4} to construct lower solution $ u_{-} $ and upper solution $ u_{+} $. Due to the results of these two propositions, all hypothesis in the monotone iteration scheme Theorem \ref{pre:thm3} hold. Applying Theorem \ref{pre:thm3}, we get the desired solution $ u $ of (\ref{zero:eqn20}).

We now consider the necessary condition. In this case, we know that (\ref{zero:eqn20}) has a positive, smooth solution. We may assume that $ S \not\equiv 0 $ since otherwise we do not need to do anything. Integrating both sides of (\ref{zero:eqn20}),
\begin{equation*}
\int_{M} -a\Delta_{g} u \dvol = \int_{M} S u^{p-1} \dvol \Rightarrow \int_{M} S u^{p-1} \dvol = 0.
\end{equation*}
Since $ u > 0 $ on $ M $, it follows that $ S $ must changes sign. Multiple $ u^{1 - p} $ on both sides of (\ref{zero:eqn20}) then integrating both sides, we have
\begin{equation*}
\int_{M} S \dvol = \int_{M} -a\Delta_{g} u \cdot u^{1 - p} \dvol = a(1 - p) \int_{M} \lvert \nabla_{g} u \rvert^{2} \cdot u^{-p} \dvol < 0.
\end{equation*}
The last inequality is due to the facts that $ u > 0 $ everywhere on $ M $ and $ (1 - p) = 1 - \frac{2n}{n - 2} < 0 $.
\end{proof}
\begin{remark}\label{zero:re2}
We can explain analytically why $ S $ must change sign. From the construction of the lower and upper solutions by our method, we observe that $ S > 0 $ somewhere in $ M $ is essential since otherwise we cannot get a nontrivial solution of the local Yamabe equation (\ref{local:eqn11}), whether the manifold is locally conformally flat nor not.
\end{remark}
\medskip

We now consider the prescribed scalar curvature problem for pointwise conformal metric on compact manifolds $ (\bar{M}, g) $ with non-empty smooth boundary, $ n = \dim M \geqslant 3 $. in our previous paper that solves the Escobar problem \cite{XU4}, we realized that the local to global method applied on closed manifolds can be transplanted to compact manifolds with non-empty boundary. We follow the same idea, and give the necessary and sufficient condition of prescribed scalar curvature problem within a pointwise conformal class $ [g] $ of the compact manifolds $ (\bar{M}, g) $ with non-empty smooth boundary $ \partial M $, $ n = \dim \bar{M} \geqslant 3 $, provided that $ \eta_{1}' = 0 $. Due to the result of the Escobar problem, the model case is a metric $ g $ which is scalar-flat and zero mean curvature.
\begin{theorem}\label{zero:thm2}
Let $ (\bar{M}, g) $ be a compact manifold with non-empty smooth boundary $ \partial M $, $ n = \dim M \geqslant 3 $. Let $ S \in \calC^{\infty}(M) $ be a given function. Assume that $ \eta_{1}' = 0 $. The function $ S $ can be realized as the prescribed scalar curvature of some conformal metric $ \tilde{g} \in [g] $ with $ h_{\tilde{g}} = 0 $ if and only if
\begin{enumerate}[(i).]
\item $ S \equiv 0 $ on $ M $;
\item $ S $ changes sign, and $ \int_{M} S \dvol < 0 $.
\end{enumerate}
\end{theorem}
\begin{proof}
Since $ \eta_{1}' = 0 $, we may assume that $ g $ is scalar-flat with minimal boundary since otherwise we can arrange a conformal change to get this in advance. Then this problem is reduced to the existence of some positive solution $ u \in \calC^{\infty}(M) $ such that
\begin{equation}\label{zero:eqn21}
-a\Delta_{g} u = S u^{p-1} \; {\rm in} \; M, \frac{\partial u}{\partial \nu} = 0 \; {\rm on} \; \partial M.
\end{equation}
We consider the sufficient condition first. The case $ S \equiv 0 $ is trivial. Assume that $ S $ satisfies (ii). We use the same function as the lower solution of (\ref{zero:eqn21}) in Proposition \ref{zero:prop3} in some interior open region $ \Omega \subset \bar{\Omega} \subset M $. The only extra thing we have to check is the boundary condition. We mention that under any conformal change $ g_{2} = \Phi^{p-2} g_{1} $, the boundary condition satisfies
\begin{equation*}
\frac{\partial \left( \Phi^{-1} u\right)}{\partial \nu_{g_{2}}} = \Phi^{-\frac{p}{2}} \frac{\partial u}{\partial \nu_{g_{1}}} = 0.
\end{equation*}
Thus we have Neumann condition for every metric in the same conformal class. It is straightforward to check the boundary condition now, since $ u_{-} \equiv 0 $ near the boundary $ \partial M $. Note that $ u_{-} \in \calC^{0}(\bar{M}) \cap H^{1}(M, g) $, nonnegative and not identically zero, such that
\begin{equation*}
-a\Delta_{g} u_{-} \leqslant S u_{-}^{p-1} \; {\rm in} \; \partial M, \frac{\partial u_{-}}{\partial \nu} = 0 \; {\rm on} \; \partial M
\end{equation*}
holds in the weak sense.

For upper solution, we have to adjust result in Corollary \ref{zero:cor1} to get a good ``almost" upper solution except within the region $ \Omega $. Since $ \int_{M} S \dvol < 0 $, we can take some $ \gamma_{0} > 0 $ so that
\begin{equation*}
\gamma : = \frac{1}{\text{Vol}_{g}(M)} \int_{M} \left( S + \gamma_{0} \right) < 0.
\end{equation*}
By standard elliptic theory, the PDE
\begin{equation*}
-a\Delta_{g} v_{0} = ( 2- p) \left( S + \gamma_{0} \right) - (2 - p) \gamma \; {\rm in} \; M, \frac{\partial v_{0}}{\partial \nu} = 0 \; {\rm on} \; M
\end{equation*}
has a solution $ v_{0} \in \calC^{\infty}(\bar{M}) $. We take $ C > 0 $ large enough, so that
\begin{equation*}
v : = v_{0} + C > 0 \; {\rm on} \; \partial M, \frac{a(p - 1)}{p - 2}\frac{\lvert \nabla_{g} v \rvert^{2}}{v} < (2 - p) \gamma.
\end{equation*}
It follows that the positive function $ v \in \calC^{\infty}(\bar{M}) $ satisfies
\begin{equation*}
-a\Delta_{g} v + \frac{a(p - 1)}{p - 2}\frac{\lvert \nabla_{g} v \rvert^{2}}{v} \leqslant (2 - p) \left( S + \gamma_{0} \right) \; {\rm in} \; M, \frac{\partial v}{\partial \nu} = 0 \; {\rm on} \; \partial M.
\end{equation*}
Set $ u = v^{\frac{1}{2 - p}} $, we conclude that $ u > 0 $ satisfies
\begin{equation}\label{zero:eqn22}
-a\Delta_{g} u \geqslant \left( S + \gamma_{0} \right) u^{p-1} \; {\rm in} \; M, \frac{\partial u}{\partial \nu} = 0 \; {\rm on} \; \partial M.
\end{equation}
The argument for (\ref{zero:eqn22}) follows from Proposition \ref{zero:prop1} exactly. Then we apply the same argument in Proposition \ref{zero:prop4} to construct an upper solution $ u_{+} $ such that $ u_{+} \in \calC^{\infty}(\bar{M}) $ is positive with $ u_{+} \geqslant u_{-} \geqslant 0 $ pointwise on $ \bar{M} $, and
\begin{equation*}
-a\Delta_{g} u_{+} \geqslant S u_{+}^{p-1} \; {\rm in} \; M, \frac{\partial u_{+}}{\partial \nu} = 0 \; {\rm on} \; \partial M
\end{equation*}
holds strongly. Thus by the monotone iteration scheme in Theorem \ref{pre:thm4}, we conclude that there exists a positive function $ u \in \calC^{\infty}(\bar{M}) $ that solves (\ref{zero:eqn21}). Note that the regularity of $ u $ is due to \cite{Che}.

For necessary condition, the case $ S \equiv 0 $ is trivial. Assume that $ S \not\equiv 0 $, we integrate (\ref{zero:eqn21}) on both sides,
\begin{equation*}
\int_{M} S u^{p-1} \dvol = \int_{M} -a\Delta_{g} u \dvol = - a \int_{\partial M} \frac{\partial u}{\partial \nu} dS_{g} = 0.
\end{equation*}
This holds since $ u > 0 $ on $ \bar{M} $. Hence $ S $ must change sign. Multiply both sides of (\ref{zero:eqn21}) by $ u^{1 - p} $ and integrate, Neumann condition implies that
\begin{equation*}
\int_{M} S \dvol = \int_{M} - a\Delta_{g} u \cdot u^{1 - p} \dvol = a (1 - p) \int_{M} \lvert \nabla_{g} u \rvert^{2} u^{-p} \dvol < 0.
\end{equation*}
\end{proof}
\medskip

\section{Prescribing Non-Constant Scalar and Mean Curvature Functions with Zero First Eigenvalue}
In \S4, we have shown in Theorem \ref{zero:thm2} the necessary and sufficient condition for prescribed scalar curvature problem with minimal boundary on compact manifolds $ (\bar{M}, g) $ with non-empty smooth boundary, $ n = \dim \bar{M} \geqslant 3 $, provided that the first eigenvalue $ \eta_{1}' $ of the conformal Laplacian is zero. Any prescribed scalar curvature functions $ S $ must satisfy $ \int_{M} S \dvol < 0 $ and change sign, or identically zero. In this section, we generalize our results to the non-trivial mean curvature scenario, provided that $ \eta_{1}' = 0 $ still. We will discuss one general case in principle:
\medskip

{\it{Given functions $ S, H \in \calC^{\infty}(\bar{M}) $ such that $ S > 0 $ somewhere and $ H \not\equiv 0 $ on $ \partial M $, when $ S $ and $ H \bigg|_{\partial M} $ can be realized as prescribed scalar and mean curvature functions, respectively}}.
\medskip

Due to the Escobar problem, we may assume $ R_{g} = h_{g} = 0 $ for our initial metric $ g $ for the whole section. Thus the problem is again reduced to the existence of some positive, smooth solution of the following PDE
\begin{equation}\label{zerog:eqn1}
-a\Delta_{g} u = Su^{p-1} \; {\rm in} \; M, \frac{\partial u}{\partial \nu}  = \frac{2}{p-2} H u^{\frac{p}{2}} \; {\rm on} \; \partial M.
\end{equation}
The question above contains the cases like both $ S $ and $ H $ change sign, or $ S > 0 $ everywhere and $ H < 0 $ everywhere, etc. Although the general case is much more flexible than the minimal boundary case, there still have some restrictions on the choices of $ S $ and $ H $ in terms of their relations. 

Let's consider the Riemann surfaces as the inspiration, as the problem of prescribed scalar and mean curvature functions is a high dimensional generalization of the prescribing Gauss and geodesic curvature functions on Riemann surface $ (\bar{N}, g) $ with boundary $ \partial N $, provided that the Euler characteristic $ \chi(\bar{N}) = 0 $. In 2-dimensional case, Gauss-Bonnet Theorem says
\begin{equation}\label{zerog:eqn2}
\int_{N} K_{g} \dvol + \int_{\partial N} \sigma_{g} dA_{g} = 2\pi \chi(\bar{N}).
\end{equation}
Here $ K_{g} $ is the Gauss curvature, $ \sigma_{g} $ is the geodesic curvature, $ dA_{g} $ is the volume form on $ \partial N $. When $ \chi(\bar{N}) = 0 $, it follows from (\ref{zerog:eqn2}) that
\begin{equation}\label{zerog:eqn3}
\int_{N} K_{g} \dvol  = - \int_{\partial N} \sigma_{g} dA_{g}.
\end{equation}
It is worthwhile to show the fact that the Euler characteristic for Riemann surfaces $ (\bar{N}, g) $ is invariant under conformal change. As an example, we show the case when $ \chi(\bar{N}) < 0 $, the same argument apply equally for the other two cases. By the uniformization theorem, we may assume that $ K_{g} = -1 $ and $ \sigma_{g} = 0 $, provided that $ \chi(\bar{N}) < 0 $. Furthermore, by Gauss-Bonnet
\begin{equation*}
- \text{Vol}_{g}(\bar{N}) = \int_{N} K_{g} \dvol = 2\pi \chi(\bar{N}).
\end{equation*}
Under the pointwise conformal change $ \tilde{g} = e^{2f} g $ with Gauss and geodesic curvatures $ K_{\tilde{g}} $ and $ \sigma_{\tilde{g}} $ with respect to $ \tilde{g} $, the following PDE holds:
\begin{equation}\label{zerog:eqn4}
-\Delta_{g} f - 1 = K_{\tilde{g}} e^{2f}  \; {\rm in} \; N, \frac{\partial f}{\partial \nu} = \sigma_{\tilde{g}} e^{f} \; {\rm on} \; \partial N.
\end{equation}
Integrating both sides with the application of (\ref{zerog:eqn3}),
\begin{align*}
& \int_{N} K_{\tilde{g}} e^{2f} \dvol = \int_{N} -\Delta_{g} f \dvol + 2\pi \chi(\bar{N}) = \int_{\partial N} - \frac{\partial f}{\partial \nu} dA_{g} + 2\pi \chi(\bar{N}) \\
\Rightarrow & \int_{N} K_{\tilde{g}} e^{2f} \dvol = -\int_{\partial N} \sigma_{\tilde{g}} e^{f} dA_{g} + 2\pi \chi(\bar{N}) \Rightarrow \int_{N} K_{\tilde{g}} d\text{Vol}_{\tilde{g}} = - \int_{\partial N} \sigma_{\tilde{g}} dA_{\tilde{g}}  + 2\pi \chi(\bar{N}) \\
\Rightarrow & 2 \pi \chi(\bar{N}) = \int_{N} K_{\tilde{g}} d\text{Vol}_{\tilde{g}} + \int_{\partial N} \sigma_{\tilde{g}} dA_{\tilde{g}}.
\end{align*}
Thus the Euler characteristic is invariant under conformal change, so is the Gauss-Bonnet theorem.
\medskip

Roughly speaking, (\ref{zerog:eqn3}) indicates that in general the sign of $ K_{g} $ should be opposite to the sign of $ \sigma_{g} $. More precisely, the sign of the average of $ K_{g} $ should be opposite to the sign of the average of $ \sigma_{g} $, or both are identically zero. We are looking for the analogy of (\ref{zerog:eqn3}) for manifolds with dimensions at least 3, provided that $ \eta_{1}' = 0 $. Since the relation (\ref{zerog:eqn3}) is obtained between the conformal change for which the Euler characteristics and the Gauss-Bonnet are invariant, it is natural to consider this type of relations under conformal change. However, due to the lack of the Gauss-Bonnet theorem in high dimensional case, we can only have an inequality as a necessary condition for prescribing scalar and mean curvature functions for some Yamabe metric, provided that $ \eta_{1}' = 0 $.

Recall that the model case for $ \eta_{1}' = 0 $ is $ R_{g} = 0 $ and $ h_{g} = 0 $. Assume that there exists a Yamabe metric $ \tilde{g} = u^{p-2} g $ with $ R_{\tilde{g}} = S $ and $ h_{\tilde{g}} = H $. Then the function $ u > 0 $ satisfies (\ref{zerog:eqn1}). In addition, we have the following relation under conformal change,
\begin{equation}\label{zerog:eqn5}
d\text{Vol}_{\tilde{g}} = u^{p} \dvol, dS_{\tilde{g}} = u^{\frac{p+2}{2}} dS_{g}.
\end{equation}
Pairing (\ref{zerog:eqn1}) by $ u $, integrating on both sides and using (\ref{zerog:eqn5}),
\begin{align*}
& \int_{M} S u^{p} \dvol = a\int_{M} - \Delta_{g}u \cdot u \dvol = -a \int_{\partial M} \frac{\partial u}{\partial \nu} \cdot u dS_{g} + a \int_{M} \lvert \nabla_{g} u \rvert^{2} \dvol \\
\Rightarrow & \int_{M} Su^{p} \dvol = -a \cdot \frac{2}{p-2} \int_{\partial M} H u^{\frac{p+2}{2}} dS_{g} + a \int_{M} \lvert \nabla_{g} u \rvert^{2} \dvol \\
\Rightarrow & \int_{M} S d\text{Vol}_{\tilde{g}} = -2(n - 1) \int_{\partial M} H dS_{\tilde{g}} +  a \int_{M} \lvert \nabla_{g} u \rvert^{2} \dvol \geqslant -2(n - 1) \int_{\partial M} H dS_{\tilde{g}}.
\end{align*}
It is straightforward to check that the model case for $ \eta_{1} = 0 $ satisfies the inequality above. In general, we conclude that
\begin{proposition}\label{zerog:prop1}
Let $ (\bar{M}, g) $ be a compact manifold with smooth boundary $ \partial M $, $ n = \dim \bar{M} \geqslant 3 $. If $ \eta_{1}' = 0 $, then a necessary condition of scalar curvature $ R_{g} $ and mean curvature $ h_{g} $ is
\begin{equation}\label{zerog:eqn6}
\int_{M} R_{g} \dvol \geqslant -2(n - 1) \int_{\partial M} h_{g} dS_{g}.
\end{equation} 
\end{proposition}
\begin{proof}
One proof is given above. An alternative way to see this is to pair the eigenfunction $ \varphi $ to both sides of the eigenvalue problem
\begin{equation*}
-a\Delta_{g} \varphi + R_{g} \varphi = 0 \; {\rm in} \; M, \frac{\partial \varphi}{\partial \nu} + \frac{2}{p-2} h_{g} \varphi = 0 \; {\rm on} \; \partial M
\end{equation*}
and integrate.
\end{proof}
Analogous to the Riemann surface's case, the necessary condition (\ref{zerog:eqn6}) indicates somehow that the sign of $ R_{g} $ should be opposite to the sign of $ h_{g} $. It is clear when $ R_{g} < 0 $. When $ R_{g} > 0 $, we integrate (\ref{zerog:eqn1}),
\begin{equation*}
\int_{M} S u^{p-1} \dvol = -a\int_{M} \Delta_{g} u \dvol = -a \cdot \frac{2}{p-2} \int_{\partial M} H u^{\frac{p}{2}} dS_{g} = -2(n - 1) \int_{\partial M} H u^{\frac{p}{2}} dS_{g}.
\end{equation*}
Roughly speaking, when $ R_{g} > 0 $, then $ h_{g} $ must be negative somewhere. We also mention that (\ref{zerog:eqn6}) is a Kazdan-Warner type restriction, but it is just a one side control; in addition, this restriction involves the choice of the conformal factor so it is different from the restriction we given in \S4. We will give analytic conditions for the functions $ S, H $, both pointwise and in average with respect to the initial metric $ g $.
\medskip

We are using the monotone iteration scheme again. As in Lemma \ref{zero:lemma1}, we convert the upper solution of (\ref{zerog:eqn1}) for given functions $ S, H $ into another PDE-type inequalities. 
\begin{lemma}\label{zerog:lemma1}
Let $ (\bar{M}, g) $ be a compact manifold with non-empty smooth boundary $ \partial M $, $ n \geqslant 3 $. Let $ S, H \in \calC^{\infty}(\bar{M}) $ be given functions. Then there exists some positive function $ u \in \calC^{\infty}(\bar{M}) $ satisfying
\begin{equation}\label{zerog:eqn7}
-a\Delta_{g} u \geqslant S u^{p-1} \; {\rm in} \; M, \frac{\partial u}{\partial \nu} \geqslant \frac{2}{p -2 } H u^{\frac{p}{2}} \; {\rm on} \; \partial M
\end{equation}
if and only if there exists some positive function $ w \in \calC^{\infty}(\bar{M}) $ satisfying
\begin{equation}\label{zerog:eqn8}
-a\Delta_{g} w + \frac{(p - 1)a}{p - 2} \cdot \frac{\lvert \nabla_{g} w \rvert^{2}}{w} \leqslant (2 - p)S \; {\rm in} \; M, \frac{\partial w}{\partial \nu} \leqslant -2H w^{\frac{1}{2}}.
\end{equation}
Moreover, the equality in (\ref{zerog:eqn7}) holds if and only if the equality in (\ref{zerog:eqn8}) holds.
\end{lemma}
\begin{proof}
We assume (\ref{zerog:eqn7}) first. Denote
\begin{equation*}
w = u^{2 - p}.
\end{equation*}
The proof of the inequalities in the interior $ M $ is exactly the same as in Lemma \ref{zero:lemma1}. For the boundary condition, we observe that $ u = w^{\frac{1}{2 - p}} $ and $ p = \frac{2n}{n - 2} $, it follows that
\begin{align*}
& \frac{\partial u}{\partial \nu} \geqslant \frac{2}{p -2 } H u^{\frac{p}{2}} \Leftrightarrow \frac{1}{2 - p} w^{\frac{1}{2 - p} - 1} \frac{\partial w}{\partial \nu} \geqslant \frac{2}{p - 2} H w^{\frac{p}{2(2 - p)}} \\
\Leftrightarrow & -\frac{n - 2}{4} w^{-\frac{n}{4} - \frac{1}{2}} \frac{\partial w}{\partial \nu} \geqslant \frac{n-2}{2} H w^{-\frac{n}{4}} \Leftrightarrow \frac{\partial w}{\partial \nu} \leqslant -2H w^{\frac{1}{2}}.
\end{align*}
The equality holds if and only if all inequalities above are equalities. For the other direction, we assume (\ref{zerog:eqn8}). Denote
\begin{equation*}
u = w^{\frac{1}{2 - p}}.
\end{equation*}
We will obtain (\ref{zerog:eqn7}). We omit the details. The same argument applies for equalities.
\end{proof}
\medskip

We now consider the first case: $ S < 0 $ somewhere with $ \int_{M} S \dvol < 0 $, and $ H \not\equiv 0 $ on $ \partial M $ with $ \int_{\partial M} H dS_{g} > 0 $. By the relation (\ref{zerog:eqn6}), the condition $ \int_{\partial M} H dS_{g} > 0 $ provides the most flexibility for the choices of $ S $.
\begin{theorem}\label{zerog:thm1}
Let $ (\bar{M}, g) $ be a compact manifold with non-empty smooth boundary $ \partial M $, $ n \geqslant 3 $. Let $ S, H \in \calC^{\infty}(\bar{M}) $ be given functions. Assume $ \eta_{1}' = 0 $. If the functions $ S, H $ satisfy
\begin{equation}\label{zerog:eqn9}
\begin{split}
& S \; \text{changes sign in $ M $}, \; \int_{M} S \dvol < 0; \\
& \int_{\partial M} H dS_{g} > 0,
\end{split}
\end{equation}
then there exists some pointwise conformal metric $ \tilde{g} \in [g] $ such that $ R_{\tilde{g}} = S $ and $ h_{\tilde{g}} = cH \bigg|_{\partial M} $ for some small enough constant $ c > 0 $. Equivalently, the PDE (\ref{zerog:eqn1}) has a positive, smooth solution for the given $ S, cH $ satisfying (\ref{zerog:eqn9}).
\end{theorem}
\begin{proof}
By hypotheses, we pick a point $ P $ and a neighborhood $ O $ containing $ P $ such that $ S > 0 $ in $ O \subset \bar{O} \subset M $. According to the conformal invariance of the conformal Laplacian in (\ref{local:eqn23}), we apply some conformal change $ g_{0} = v^{p-2} g $ as in Proposition \ref{zero:prop2}, such that $ R_{g_{0}} < 0 $ in some open subset $ \Omega \subset O $. By the same argument in Proposition \ref{zero:prop3}, in which we used either Proposition \ref{local:prop2} or Proposition \ref{local:prop3},  the following PDE
\begin{equation}\label{zerog:eqn10a}
-a\Delta_{g_{0}} \tilde{u} + R_{g_{0}} \tilde{u} = S \tilde{u}^{p-1} \; {\rm in} \; \Omega, \tilde{u} = 0 \; {\rm on} \; \partial M
\end{equation}
has a positive solution $ \tilde{u} \in \calC^{\infty}(\Omega) \cap \calC^{0}(\bar{\Omega}) \cap H_{0}^{1}(\Omega, g_{0}) $ by shrinking $ \Omega $ further if necessary. Note that under any conformal change $ g_{2} = \Phi^{p-2} g_{1} $, the boundary condition satisfies
\begin{equation*}
B_{g_{2}} \left( \Phi^{-1} u \right) = \Phi^{-\frac{p}{2}} B_{g_{1}} u.
\end{equation*}
Thus the lower solution for the PDE
\begin{equation}\label{zerog:eqn10}
-a\Delta_{g_{0}} u + R_{g_{0}} u = S u^{p-1} \; {\rm in} \; M, \frac{\partial u}{\partial \nu_{g_{0}}} = cH u^{\frac{p}{2}} \; {\rm on} \; \partial M
\end{equation}
is given by
\begin{equation}\label{zerog:eqn11}
u_{-} : = \begin{cases} \tilde{u}, & \; {\rm in} \; \Omega \\ 0, & \; {\rm in} \; \bar{M} \backslash \bar{\Omega} \end{cases}
\end{equation}
in the weak sense. The interior part is the same as in Proposition \ref{zero:prop3} or see \cite[Thm.~4.4]{XU3}. The boundary condition is trivial to check since $ u_{-} \equiv 0 $ on a collar region of $ \partial M $. 

We now construct a good candidate of the upper solution. Since $ \int_{M} S \dvol < 0 $, we can choose two small enough constants $ \gamma, \gamma' > 0 $ such that $ \int_{M} (S + \gamma + \gamma' )\dvol < 0 $ still. Take the constant $ \gamma'' < 0 $ such that
\begin{equation*}
\frac{2- p}{a} \int_{M} (S + \gamma + \gamma' )\dvol = - \int_{\partial M} \gamma'' dS_{g}.
\end{equation*}
By the standard linear elliptic PDE theory, the equality above implies that the following PDE
\begin{equation*}
-a \Delta_{g} \phi = (2 - p) (S + \gamma + \gamma') \; {\rm in} \; M, \frac{\partial \phi}{\partial \nu} = \gamma'' \; {\rm on} \; \partial M
\end{equation*}
has a smooth solution $ \phi $. We define
\begin{equation*}
\tilde{\phi} = \phi + C
\end{equation*}
for large enough $ C $ such that
\begin{equation}\label{zerog:eqn12}
\tilde{\phi} > 0 \; {\rm on} \; \bar{M}, \frac{(p - 1)a}{p - 2} \frac{\lvert \nabla_{g} \tilde{\phi} \rvert^{2}}{\tilde{\phi}} + (2 - p) \gamma'' < 0 \; {\rm in} \; M.
\end{equation}
The condition (\ref{zerog:eqn12}) implies that
\begin{equation*}
-a\Delta_{g} \tilde{\phi} + \frac{(p - 1)a}{p - 2} \frac{\lvert \nabla_{g} \tilde{\phi} \rvert^{2}}{\tilde{\phi}}  < (2 - p) ( S + \gamma') \; {\rm in} \; M.
\end{equation*}
Fix this $ \tilde{\phi} $. We now choose the constant $ c > 0 $ small enough such that
\begin{equation*}
\frac{\partial \tilde{\phi}}{\partial \nu} = \frac{\partial \phi}{\partial \nu} = \gamma'' \leqslant -2cH \tilde{\phi}^{\frac{1}{2}}.
\end{equation*}
It can be done since $ \gamma'' < 0 $. By Lemma \ref{zerog:lemma1}, the function $ \tilde{\varphi} = \tilde{\phi}^{\frac{1}{2 - p}} $ satisfies
\begin{equation*}
-a\Delta_{g} \tilde{\varphi} \geqslant (S + \gamma') \tilde{\varphi}^{p - 1} \; {\rm in} \; M, \frac{\partial \tilde{\varphi}}{\partial \nu} \geqslant \frac{2}{p-2} cH \tilde{\varphi}^{\frac{p}{2}} \; {\rm on} \; \partial M.
\end{equation*}
Applying the conformal change $ g_{0} = v^{p-2} g $, the function $ \varphi : = v \cdot \tilde{\varphi} > 0 $ on $ \bar{M} $ satisfies
\begin{equation}\label{zerog:eqn13}
-a\Delta_{g_{0}} \varphi + R_{g_{0}} \varphi \geqslant (S + \gamma') \varphi^{p - 1} \; {\rm in} \; M, \frac{\partial \varphi}{\partial \nu} \geqslant \frac{2}{p-2} cH \varphi^{\frac{p}{2}} \; {\rm on} \; \partial M
\end{equation}
due to the conformal invariance of the conformal Laplacian as well as the boundary condition. 

By the same argument in Proposition \ref{zero:prop4}, we apply the solution of the local Yamabe solution in (\ref{zerog:eqn10a}) and the function $ \varphi $ in (\ref{zerog:eqn13}) to conclude that there exists a positive function $ u_{+} \in \calC^{\infty}(\bar{M}) $ such that
\begin{equation}\label{zerog:eqn14}
-a\Delta_{g_{0}} u_{+} + R_{g_{0}} u_{+} \geqslant (S + \gamma') u_{+}^{p - 1} \; {\rm in} \; M, \frac{\partial u_{+}}{\partial \nu} \geqslant \frac{2}{p-2} cH u_{+}^{\frac{p}{2}} \; {\rm on} \; \partial M.
\end{equation}
Furthermore, $ 0 \leqslant u_{-} \leqslant u_{+} $. Note that (\ref{zerog:eqn14}) holds for all smaller constants $ c > 0 $. We therefore shrink $ c $ if necessary so that the condition (\ref{pre:eqn17}) holds for the function $ cH $. We apply the monotone iteration scheme in Theorem \ref{pre:thm4} to conclude the existence of a positive, smooth solution $ u $ of (\ref{zero:eqn10}). Due to conformal invariance of the conformal Laplacian and the Robin boundary condition, the function $ v^{-1} u $ solves (\ref{zerog:eqn1}) for $ S $ and $ cH $. Note that the last step holds since $ \eta_{1}' = 0 $.
\end{proof}
\medskip

Analogously, we would like to consider the following two cases:
\begin{enumerate}[(i).]
\item $ S < 0 $ somewhere with $ \int_{M} S \dvol < 0 $, and $ H \not\equiv 0 $ on $ \partial M $ with $ \int_{\partial M} H dS_{g} = 0 $;
\item $ S < 0 $ somewhere with $ \int_{M} S \dvol < 0 $, and $ H \not\equiv 0 $ on $ \partial M $ with $ \int_{\partial M} H dS_{g} < 0 $.
\end{enumerate}
Note that although no direct reasoning forces the function $ S $ to change sign, we can see from the case (ii) that if $ H < 0 $ everywhere, then the inequality (\ref{zerog:eqn6}) implies that $ S $ must be positive somewhere also. Since we have assumed that $ \int_{M} S \dvol < 0 $, then $ S $ must change sign. We have the following two results.
\begin{corollary}\label{zerog:cor1}
Let $ (\bar{M}, g) $ be a compact manifold with non-empty smooth boundary $ \partial M $, $ n \geqslant 3 $. Let $ S, H \in \calC^{\infty}(\bar{M}) $ be given functions. Assume $ \eta_{1}' = 0 $. If the functions $ S, H $ satisfy
\begin{equation}\label{zerog:eqn15}
\begin{split}
& S \; \text{changes sign in $ M $}, \; \int_{M} S \dvol < 0; \\
& \int_{\partial M} H dS_{g} = 0,
\end{split}
\end{equation}
then there exists some pointwise conformal metric $ \tilde{g} \in [g] $ such that $ R_{\tilde{g}} = S $ and $ h_{\tilde{g}} = cH \bigg|_{\partial M} $ for some small enough constant $ c > 0 $. 
\end{corollary}
\begin{proof}
It is essentially the same as the proof in Theorem \ref{zerog:thm1}.
\end{proof}
\begin{corollary}\label{zerog:cor2}
Let $ (\bar{M}, g) $ be a compact manifold with non-empty smooth boundary $ \partial M $, $ n \geqslant 3 $. Let $ S, H \in \calC^{\infty}(\bar{M}) $ be given functions. Assume $ \eta_{1}' = 0 $. If the functions $ S, H $ satisfy
\begin{equation}\label{zerog:eqn16}
\begin{split}
& S \; \text{changes sign in $ M $}, \; \int_{M} S \dvol < 0; \\
& \int_{\partial M} H dS_{g} < 0,
\end{split}
\end{equation}
then there exists some pointwise conformal metric $ \tilde{g} \in [g] $ such that $ R_{\tilde{g}} = S $ and $ h_{\tilde{g}} = cH \bigg|_{\partial M} $ for some small enough constant $ c > 0 $. 
\end{corollary}
\begin{proof}
It is essentially the same as the proof in Theorem \ref{zerog:thm1}.
\end{proof}
\begin{remark}\label{zerog:re1} According to the previous three results, we see that as long as $ \int_{M} S \dvol < 0 $ and $ S $ changes sign, the sign of $ \int_{\partial M} H dS_{g} $ does not matter. We discussed above that we can only see why $ S $ must changes sign when $ H < 0 $ everywhere on $ \partial M $ or $ H \equiv 0 $ on $ \partial M $. Similarly, we have no direct uniform reason to see why we must assume $ \int_{M} S \dvol < 0 $. However, we can see this as a necessary condition if we have $ H > 0 $ everywhere on $ \partial M $ or when $ H \equiv 0 $ on $ \partial M $. The latter case $ H \equiv 0 $ is shown in \S4, and it associates with the case $ \int_{M} H \dvol = 0 $. 

When $ H > 0 $ everywhere, which associates with the case $ \int_{\partial M} H dS_{g} > 0 $, we assume the PDE (\ref{zerog:eqn1}) holds for some non-constant positive, smooth function $ u $, since if $ u $ is a constant then it reduces to the case $ S = H \equiv 0 $, the model case. We multiple both sides by $ u^{1 - p} $ and integrate, we have
\begin{align*}
\int_{M} S \dvol & = -a \int_{M} u^{1 - p} \Delta_{g} u \dvol = - a \int_{\partial M} u^{1 - p} \frac{\partial u}{\partial \nu} dS_{g} + a \int_{M} ( 1- p) u^{-p} \lvert \nabla_{g} v \rvert^{2} \dvol \\
& = - \frac{2a}{ p -2} \int_{\partial M} H u^{1 - \frac{p}{2}} dS_{g} + a( 1- p) \int_{M} u^{-p} \lvert \nabla_{g} v \rvert^{2} \dvol.
\end{align*}
If $ H > 0 $ everywhere then the first term in the last line above is negative as $ -\frac{2a}{ p -2} = -2(n - 1) < 0, \forall n \geqslant 3 $. Since $ u $ is non-constant, the second term in the last line above is also negative as $ 1 - p = 1 - \frac{2n}{n - 2} < 0, \forall n \geqslant 3 $. Thus we conclude that $ \int_{M} S \dvol < 0 $.

Since we have full flexibility to choose the prescribed mean curvature function $ H $ when $ \int_{M} S \dvol < 0 $ and $ S $ changes sign, we would like to conjecture, although we only have partial reasoning, that the prescribed scalar curvature function must satisfy the two conditions: $ \int_{M} S \dvol < 0 $ and $ S $ changes sign, even when the mean curvature function is not identically zero, provided that $ \eta_{1}' = 0 $.
\end{remark}
\medskip

\section{Prescribing Gauss and Geodesic Curvature Problem on Compact Riemann Surfaces with Boundary with Zero Euler Characteristic}
Throughout this section, $ (\bar{M}, g) $ is denoted to be the $ 2 $-dimensional compact Riemann surface with non-empty smooth boundary $ \partial M $ and unit outward normal vector field $ \nu $ along $ \partial M $; throughout this section, the Gauss and geodesic curvatures of $ g $ are denoted to be $ K_{g} $ and $ \sigma_{g} $, respectively. In this section, we try to extend our necessary and sufficient condition of prescribing scalar curvature problem on compact manifolds with boundary with dimensions at least $ 3 $ to $ (\bar{M}, g) $, in terms of prescribing Gauss curvature function and zero geodesic curvature functions for some conformal metric $ \tilde{g} \in [g] $. Instead of the local-to-global analysis and monotone iteration schemes, we apply global variational method here, which is inspired by Kazdan and Warner \cite{KW2}. We then discuss the prescribing Gauss and geodesic curvature problems on $ (\bar{M}, g) $, especially for non-trivial geodesic curvature functions.

The model case when $ \chi(\bar{M}) = 0 $ is $ K_{g} \equiv \sigma_{g} \equiv 0 $, by uniformization theorem and Gauss-Bonnet Theorem. From now on, we always assume our initial metric $ g $ has zero Gauss and zero geodesic curvatures. Given two functions $ K, \sigma \in \calC^{\infty}(\bar{M}) $, the existence of a conformal metric $ \tilde{g} = e^{2u} g $ for some $ u \in \calC^{\infty}(\bar{M}) $ such that $ K_{\tilde{g}} = K $ and $ \sigma_{\tilde{g}} = \sigma $ is reduced to the existence of some smooth solution of the following PDE
\begin{equation}\label{de2:eqn1}
-\Delta_{g} u = K e^{2u} \; {\rm in} \; M, \frac{\partial u}{\partial \nu} = \sigma e^{u} \; {\rm on} \; \partial M.
\end{equation}
When we discuss the necessary and sufficient conditions of prescribing Gauss curvature problem with zero geodesic curvature, the boundary condition in (\ref{de2:eqn1}) is further reduced to the Neumann condition, i.e.
\begin{equation}\label{de2:eqn2}
-\Delta_{g} u = K e^{2u} \; {\rm in} \; M, \frac{\partial u}{\partial \nu} = 0 \; {\rm on} \; \partial M.
\end{equation}
Let's discuss the necessary condition by assuming the existence of some solution $ u \in \calC^{\infty}(\bar{M}) $ of (\ref{de2:eqn2}) first. The relation in (\ref{zero:eqn3}) indicates that
\begin{equation*}
\int_{M} K d\text{Vol}_{\tilde{g}} = 0
\end{equation*}
for the metric $ \tilde{g} = e^{2u} g $. Therefore $ K $ must changes sign or identically zero. Multiple $ e^{-2u} $ on both sides of (\ref{de2:eqn2}) and integrate, we have
\begin{equation*}
\int_{M} K \dvol = -2\int_{M} e^{-2u} \lvert \nabla_{g} u \rvert^{2} \dvol.
\end{equation*}
The inequality above is negative unless the function $ u $ is some constant, which is the case exactly when $ K \equiv 0 $. Therefore the necessary condition of prescribing Gauss curvature function $ K $ for the some conformal metric $ \tilde{g} \in [g] $ with zero geodesic curvature is either $ K \equiv 0 $ or $ K $ changes sign and $ \int_{M} K \dvol < 0 $.
\medskip

We would like to show that the condition given above is exactly the sufficient condition also. It reduces to solve the PDE (\ref{de2:eqn2}) for the function $ K $ satisfying the condition above. We apply the variational method. It is standard to handle the weak solution of (\ref{de2:eqn2}) in the standard Sobolev space $ H^{1}(M, g) $ since the weak form of (\ref{de2:eqn2}) can identify the PDE and the boundary condition by using different test functions, provided that the solution is regular enough, say, at least $ \calC^{2}(\bar{M}) $. We observe that the main issue is to control $ e^{2u} $ term, for its largeness in an appropriate functional space.

In $ 2 $-dimensional case with exponential function, we need the idea of Moser and Trudinger. We start with the definition of a different Hilbert space, which is a subspace of $ H^{1}(M, g) $.
\begin{definition}\label{de2:def1}
Let $ (\bar{M}, g) $ be a compact manifold with non-empty smooth boundary, $ n = \dim \bar{M} \geqslant 1 $. We define
\begin{equation*}
H_{\perp}^{1}(M, g) : = \lbrace u \in H^{1}(M, g) : \int_{M} u \dvol = 0 \rbrace.
\end{equation*}
\end{definition}
It is standard to know that $ H_{\perp}^{1}(M, g) $ is a Hilbert space.

The next two results are some variation of Trudinger's inequality and a consequence of the Trudinger's inequality.
\begin{proposition}\label{de2:prop1}\cite[Formula.~4.15]{T3}
Let $ (\bar{M}, g) $ be a compact manifold with non-empty smooth boundary, $ n = \dim \bar{M} \geqslant 1 $. Then there exists a natural inclusion $ \imath $ such that
\begin{equation*}
\imath : H^{\frac{n}{2}}(M, g) \rightarrow \calL^{q}(M, g), \forall q \in [1, \infty).
\end{equation*}
The inclusion map $ \imath $ is compact. In addition, we have
\begin{equation}\label{de2:eqn3}
\lVert u \rVert_{\calL^{q}(M, g)} \leqslant C_{q} \lVert u \rVert_{H^{\frac{n}{2}}(M, g)}, \forall u \in H^{\frac{n}{2}}(M, g).
\end{equation}
The constant $ C_{q} $ is independent of the choice of $ u $.
\end{proposition}
\begin{proposition}\label{de2:prop2}
Let $ (\bar{M}, g) $ be a compact manifold with non-empty smooth boundary, $ n = \dim \bar{M} \geqslant 1 $. Let $ \alpha \in \R $ be some constant. If here exists a sequence $ \lbrace u_{k} \rbrace $ such that $ u_{k} \rightarrow u $ weakly in $ H^{\frac{n}{2}}(M, g) $-norm, then
\begin{equation}\label{de2:eqn4}
e^{\alpha u_{k}} \rightarrow e^{\alpha u}
\end{equation}
strongly in $ \calL^{1}(M, g) $-norm.
\end{proposition}
It is well-known that the Poincar\'e inequality holds for elements in $ H_{\perp}^{1}(M, g) $, i.e.
\begin{equation}\label{de2:eqn5}
\lVert u \rVert_{\calL^{2}(M, g)} \leqslant C_{1} \lVert \nabla u \rVert_{\calL^{2}(M, g)}, \forall u \in H_{\perp}^{1}(M, g).
\end{equation}
The constant $ C_{1} $ is independent of the choice of $ u $. According to the Poincar\'e inequality and Proposition \ref{de2:prop1}, which states that the compact embedding $ H^{1}(M, g) \hookrightarrow \calL^{q}(M, g) $ holds for all $ q \in [1, \infty) $ when $ n = \dim \bar{M} = 2 $, we have the following consequences.
\begin{proposition}\label{de2:prop3}
Let $ (\bar{M}, g) $ be a compact manifold with non-empty smooth boundary, $ n = \dim \bar{M} = 2 $. Let $ u \in H_{\perp}^{1}(M, g) $ which satisfies $ \lVert \nabla u \rVert_{\calL^{2}(M, g)} \leqslant 1 $. Then there exist positive constants $ C_{2}, C_{3} $ such that
\begin{equation}\label{de2:eqn6}
\int_{M} e^{C_{2} u} \dvol \leqslant C_{3}.
\end{equation}
The constants $ C_{2}, C_{3} $ are independent of the choice of $ u $.
\end{proposition}
\begin{proposition}\label{de2:prop4}
Let $ (\bar{M}, g) $ be a compact manifold with non-empty smooth boundary, $ n = \dim \bar{M} = 2 $. Then for any function $ u \in H^{1}(M, g) $ and any positive constant $ \beta > 0 $, there exist positive constants $ C_{4}, C_{5} $ such that
\begin{equation}\label{de2:eqn7}
\int_{M} e^{\beta \lvert u \rvert} \dvol \leqslant C_{4} e^{\left( \frac{\beta}{\text{Vol}_{g}(M)} \left\lvert \int_{M} u \dvol \right\rvert+ \frac{ \beta^{2} \lVert \nabla u \rVert^{2} }{4C_{5} } \right)}.
\end{equation}
Here $ C_{4}, C_{5} $ are constants independent of the choice of $ u $.
\end{proposition}
\begin{proposition}\label{de2:prop5}
Let $ (\bar{M}, g) $ be a compact manifold with non-empty smooth boundary, $ n = \dim \bar{M} = 2 $. Then
\begin{equation}\label{de2:eqn8}
u \in H^{1}(M, g) \Rightarrow e^{u} \in \calL^{q}(M, g), \forall q \in [1, \infty).
\end{equation}
\end{proposition}
\begin{remark}\label{de2:re1}
The proofs of Proposition \ref{de2:prop3}, Proposition \ref{de2:prop4} and Proposition \ref{de2:prop5} are exactly the same as in \cite{KW2}. Roughly speaking, Proposition \ref{de2:prop3} is proven by a Taylor expansion of the exponential function, due to Trudinger; Proposition \ref{de2:prop4} follows from Proposition \ref{de2:prop3} by applying $ v = \frac{u - \frac{1}{\text{Vol}_{g}(M)}\int_{M}u \dvol}{\beta} $ in (\ref{de2:eqn6}), and then apply Young's inequality; Proposition \ref{de2:prop5} is a natural consequence of Proposition \ref{de2:prop4} by choosing appropriate constant $ \beta $ in (\ref{de2:eqn7}).
\end{remark}
\medskip

With all preparations above, we introduce the next result for sufficient conditions of prescribing Gauss curvature with $ \chi(\bar{M}) = 0 $.
\begin{theorem}\label{de2:thm1}
Let $ (\bar{M}, g) $ be a compact Riemann surface with non-empty smooth boundary $ \partial M $. Assume that $ \chi(\bar{M}) = 0 $. If the given function $ K \in \calC^{\infty}(\bar{M}) $ satisfies
\begin{equation}\label{de2:eqn9}
\begin{split}
& \text{either} \; K \equiv 0; \\
& \text{or} \; \int_{M} K \dvol < 0 \; \text{and $ K $ changes sign}, \\
\end{split}
\end{equation}
then there exists a smooth function $ u $ that solves (\ref{de2:eqn2}) with the function $ K $ given in (\ref{de2:eqn9}), i.e. there exists a metric $ \tilde{g} = e^{2u} g $ such that $ K_{\tilde{g}} = K $ and $ \sigma_{\tilde{g}} = 0 $.
\end{theorem}
\begin{proof}
When $ K \equiv 0 $, it is trivial. So assume that $ K $ is nontrivial such that (\ref{de2:eqn9}) holds. The variational method used here is essentially due to Kazdan and Warner \cite{KW2}. Consider the space
\begin{equation}\label{de2:eqn10}
B : = \lbrace u \in H_{\perp}^{1}(M, g) : \int_{M} K e^{2u} \dvol = 0 \rbrace.
\end{equation}
Since $ K $ changes sign, due to the same reasoning in \cite{KW2}, the space $ B $ is not empty, i.e. there exists at least one element $ u_{0} \in B $. Define the functional to be
\begin{equation}\label{de2:eqn11}
J(u) : = \frac{1}{2} \int_{M} \lvert \nabla_{g} u \rvert^{2} \dvol.
\end{equation}
Our goal is to minimize $ J $ within the space $ B $. $ J(u) \geqslant 0, \forall u \in B $, hence there exists a minimizing sequence $ \lbrace u_{k} \rbrace_{k \in \mathbb{N}} \in B $ such that $ J(u_{k}) $ converges from the right to $ A : = \inf_{u \in B} J(u) $. By making further choices of the sequence, if necessary, we may assume that $ J(u_{k}) \leqslant J(u_{0}), \forall k \in \mathbb{N} $. By Poincar\'e inequality, we have
\begin{equation*}
\lVert u_{k} \rVert_{H^{1}(M, g)} \leqslant \left(1 + C_{1} \right) \lVert \nabla_{g} u_{k} \rVert_{\calL^{2}(M, g)} \leqslant \left(1 + C_{1} \right) J(u_{0})^{\frac{1}{2}}, \forall k \in \mathbb{N}.
\end{equation*}
Therefore $ \lbrace u_{k} \rbrace_{k \in \mathbb{N}} $ is uniformly bounded in $ H^{1} $-norm. Standard Hilbert space theory as well as the weak compactness implies that there exists a subsequence $ \lbrace u_{k_{j}} \rbrace $ of the original sequence such that
\begin{equation}\label{de2:eqn12}
u_{k_{j}} \rightharpoonup u \; \text{weakly in $ H^{1}(M, g) $}.
\end{equation}
For simplicity of notation, we still denote the subsequence to be $ u_{k} $ in this proof. Note that by Rellich's theorem, $ u_{k} \rightarrow u $ strongly in $ \calL^{2} $-norm. We show that $ u \in H_{\perp}^{1}(M, g) $. To see this, we observe that
\begin{equation*}
\left\lvert \int_{M} u \dvol \right\rvert = \left\lvert \int_{M} \left( u - u_{k} \right) \dvol \right\rvert \leqslant \text{Vol}_{g}(M)^{\frac{1}{2}} \left( \int_{M} \lvert u_{k} - u \rvert^{2} \dvol \right)^{\frac{1}{2}}.
\end{equation*}
The last term above could be arbitrarily small by making $ k $ large enough. It follows that $ u \in H_{\perp}^{1}(M, g) $. By Proposition \ref{de2:prop2}, the weak convergence in (\ref{de2:eqn12}) implies that
\begin{equation*}
\left\lvert \int_{M} K e^{2u} \dvol \right\rvert = \left\lvert \int_{K} \left( K e^{2u} - K e^{2u_{k}} \right) \dvol \right\rvert \rightarrow 0
\end{equation*}
as $ k \rightarrow \infty $. Hence $ \int_{M} K e^{2u} \dvol = 0 $. We conclude that the weak limit $ u $ in (\ref{de2:eqn12}) is an element of $ B $.

Following exactly the same argument in \cite[Thm.~5.3]{KW2}, we conclude that
\begin{equation*}
A = J(u),
\end{equation*}
i.e. $ u $ minimizes the functional $ J $ for all elements in $ B $.

According to the variational method with the constraint, the Euler-Lagrange equation with respect to the minimizer $ u $ is of the form
\begin{equation}\label{de2:eqn13}
\int_{M} \nabla_{g} u \cdot \nabla_{g} v \dvol + \int_{M} c_{1} K e^{2u} v \dvol + \int_{M} c_{2} v \dvol = 0, \forall v \in H^{1}(M, g)
\end{equation}
for some constants $ c_{1}, c_{2} $ which will be determined later. (\ref{de2:eqn13}) implies that the weak solution $ u $ has homogeneous Neumann boundary condition as natural boundary conditions. However, (\ref{de2:eqn13}) contains no boundary terms. We consider the regularity of $ u $. Since $ u \in H^{1}(M, g) $, Proposition \ref{de2:prop5} implies that $ K e^{2u} \in \calL^{q}(M, g), \forall q \in [1, \infty) $. We take some $ q > 2 $. It follows from the $ W^{s, q} $-type elliptic regularity, see e.g. \cite[Prop.~2.2]{XU4}, that $ u \in H^{2, q}(M, g) $. Then by the standard bootstrapping argument, we conclude that $ u \in \calC^{\infty}(\bar{M}) $. It follows from (\ref{de2:eqn13}) that $ \frac{\partial u}{\partial \nu} = 0 $ in the strong sense.

To determine $ c_{2} $, we take $ v \equiv 1 $. By the fact that $ u \in B $ hence $ \int_{M} K e^{2u} \dvol = 0 $, we have $ c_{2} = 0 $. We take $ v = e^{-2u} $, it then follows that
\begin{equation*}
\int_{M} - 2e^{-2u} \lvert \nabla_{g} u \rvert^{2} \dvol + c_{1} \int_{M} K \dvol = 0.
\end{equation*}
Hence $ c_{1} < 0 $ since $ \int_{M} K \dvol < 0 $. Since $ u \in \calC^{\infty}(\bar{M}) $, we conclude that
\begin{equation}\label{de2:eqn14}
-\Delta_{g} u = -c_{1} K e^{2u} \; {\rm in} \; M, \frac{\partial u}{\partial \nu} = 0 \; {\rm on} \; \partial M
\end{equation}
holds in the classical sense. Since $ c_{1} < 0 $, we denote $ -c_{1} = e^{2\gamma} $ for some $ 2\gamma $. Denote
\begin{equation*}
\tilde{u} : = u + \gamma,
\end{equation*}
we have
\begin{equation*}
-\Delta_{g} \tilde{u} = K e^{2 \tilde{u}} \; {\rm in} \; M, \frac{\partial \tilde{u}}{\partial \nu} = 0 \; {\rm on} \; \partial M.
\end{equation*}
$ \tilde{u} $ is the desired solution of (\ref{de2:eqn2}) for $ K $ satisfying (\ref{de2:eqn9}).
\end{proof}

\medskip
\section{The Generalization of the Han-Li Conjecture}
In this section, we discuss the generalization of the Han-Li conjecture \cite{HL}.The standard Han-Li conjecture is equivalent to the existence of a positive, smooth solution of the following PDE
\begin{equation}\label{HL:eqn1}
-a\Delta_{g} u + R_{g} u = \lambda u^{p-1} \; {\rm in} \; M, \frac{\partial u}{\partial \nu} + \frac{2}{p-2} h_{g} u = \frac{2}{p-2} \zeta u^{\frac{p}{2}} \; {\rm on} \; \partial M
\end{equation}
for some constants $ \lambda, \zeta \in \R $. It was first proven in \cite{XU5}, which is listed below. Note that the constant mean curvature $ \zeta $ on $ \partial M $ is required to be positive. 
\begin{theorem}\label{HL:thm1}\cite[\S1]{XU5}
Let $ (\bar{M}, g) $ be a compact manifold with smooth boundary, $ \dim \bar{M} \geqslant 3 $. Let $ \eta_{1}' $ be the first eigenvalue of the boundary eigenvalue problem $ \Box_{g} = \eta_{1}' u $ in $ M $, $ B_{g} u = 0 $ on $ \partial M $. Then: \\
\begin{enumerate}[(i).]
\item If $ \eta_{1}' = 0 $, (\ref{HL:eqn1}) with constant functions $ S = \lambda \in \R $, $ H = \zeta \in \R $ admits a real, positive solution $ u \in \calC^{\infty}(\bar{M}) $ with $ \lambda = \zeta = 0 $;
\item If $ \eta_{1}' < 0 $, (\ref{HL:eqn1}) with constant functions $ S = \lambda \in \R $, $ H = \zeta \in \R $ admits a real, positive solution $ u \in \calC^{\infty}(\bar{M}) $ with some $ \lambda < 0 $ and $ \zeta > 0 $;
\item If $ \eta_{1}' > 0 $, (\ref{HL:eqn1}) with constant functions $ S = \lambda \in \R $, $ H = \zeta \in \R $ admits a real, positive solution $ u \in \calC^{\infty}(\bar{M}) $ with some $ \lambda > 0 $ and $ \zeta > 0 $.
\end{enumerate}
\end{theorem} 
With the aid of the new version of the monotone iteration scheme in Theorem \ref{pre:thm4}, we can extend the Han-Li conjecture by showing that on compact manifolds $ (\bar{M}, g) $ with boundary,
\begin{enumerate}[(i).]
\item If $ \eta_{1}' < 0 $, then (\ref{HL:eqn1}) admits a positive, smooth solution with some $ \lambda < 0 $ and $ \zeta < 0 $; 
\item If $ \eta_{1}' > 0 $, then (\ref{HL:eqn1}) admits a positive, smooth solution with some $ \lambda > 0 $ and $ \zeta < 0 $.
\end{enumerate}
As a prerequisite, we need a result in terms of the perturbation of negative first eigenvalue of conformal Laplacian.
\begin{proposition}\label{HL:prop1}
Let $ (\bar{M}, g) $ be a compact Riemannian manifold with non-empty smooth boundary $ \partial M $, $ n = \dim \bar{M} \geqslant 3 $. Let $ \beta > 0 $ be a small enough constant. If $ \eta_{1}' < 0 $, then the quantity
\begin{equation*}
\eta_{1, \beta}' = \inf_{u \neq 0} \frac{a\int_{M} \lvert \nabla_{g} u \rvert^{2} \dvol + \int_{M} R_{g} u^{2} \dvol + \frac{2a}{p-2} \int_{\partial M} (h_{g} + \beta) u^{2} dS}{\int_{M} u^{2} \dvol} < 0.
\end{equation*}
In particular, $ \eta_{1, \beta}' $ satisfies
\begin{equation}\label{HL:eqn2}
-a\Delta_{g} \varphi + R_{g} \varphi = \eta_{1, \beta}' \varphi \; {\rm in} \; M, \frac{\partial \varphi}{\partial \nu} + \frac{2}{p-2} (h_{g} + \beta) \varphi = 0 \; {\rm on} \; \partial M
\end{equation}
with some positive function $ \varphi \in \calC^{\infty}(\bar{M}) $.
\end{proposition}
\begin{proof}
Since $ \eta_{1}' < 0 $, the normalized first eigenfunction $ \varphi_{1} $, i.e. $ \int_{M} \varphi_{1}^{2} \dvol = 1 $, satisfies
\begin{equation*}
\eta_{1}' = a\int_{M} \lvert \nabla_{g} \varphi_{1} \rvert^{2} \dvol + \int_{M} R_{g} \varphi_{1}^{2} \dvol + \frac{2a}{p-2} \int_{\partial M} h_{g} \varphi_{1}^{2} dS
\end{equation*}
By characterization of $ \eta_{1, \beta}' $, we have
\begin{equation*}
\eta_{1, \beta}' \leqslant a\int_{M} \lvert \nabla_{g} \varphi_{1} \rvert^{2} \dvol + \int_{M} R_{g} \varphi_{1}^{2} \dvol + \frac{2a}{p-2} \int_{\partial M} (h_{g} + \beta) \varphi_{1}^{2} dS = \eta_{1}' + \beta \int_{\partial M} \varphi_{1}^{2} dS.
\end{equation*}
Since $ \varphi_{1} $ is fixed, it follows that $ \eta_{1, \beta}' < 0 $ if $ \beta > 0 $ is small enough. 
\end{proof}
The above result says that when $ \eta_{1}' < 0 $, there is a small room to allow us to perturb the boundary condition, but still keep the sign of the perturbed eigenvalue unchanged. We anticipate the same property for positive first eigenvalue case. But here we need the perturbation on the conformal Laplacian operator but not on the Robin boundary condition for positive first eigenvalue case, which was mentioned in \cite[\S5]{XU4}.
\begin{proposition}\label{HL:prop2}
Let $ (\bar{M}, g) $ be a compact Riemannian manifold with non-empty smooth boundary $ \partial M $, $ n = \dim \bar{M} \geqslant 3 $. Let $ \beta < 0 $ be a constant with small enough absolute value. If $ \eta_{1}' > 0 $, then the quantity
\begin{equation*}
\eta_{1, \beta}' = \inf_{u \neq 0} \frac{a\int_{M} \lvert \nabla_{g} u \rvert^{2} \dvol + \int_{M} \left( R_{g} + \beta \right) u^{2} \dvol + \frac{2a}{p-2} \int_{\partial M} h_{g} u^{2} dS}{\int_{M} u^{2} \dvol} > 0.
\end{equation*}
In particular, $ \eta_{1, \beta}' $ satisfies
\begin{equation}\label{HL:eqn3}
-a\Delta_{g} \varphi + \left( R_{g} + \beta \right) \varphi = \eta_{1, \beta}' \varphi \; {\rm in} \; M, \frac{\partial \varphi}{\partial \nu} + \frac{2}{p-2} h_{g} \varphi = 0 \; {\rm on} \; \partial M
\end{equation}
with some positive function $ \varphi \in \calC^{\infty}(\bar{M}) $.
\end{proposition}
\begin{proof}
Set $ \eta_{1, \beta}' = \eta_{1} + \beta $, we have
\begin{equation*}
-a\Delta_{g} \varphi + \left( R_{g} + \beta \right) \varphi = \eta_{1}' \varphi + \beta \varphi \; {\rm in} \; M, \frac{\partial \varphi}{\partial \nu} + \frac{2}{p-2} h_{g} \varphi = 0 \; {\rm on} \; \partial M.
\end{equation*}
The function $ \varphi $ is the first eigenfunction with respect to $ \eta_{1}' $. If $ \lvert \beta \rvert $ is small enough, $ \eta_{1, \beta}' > 0 $ for sure.
\end{proof}
We are ready to show the extension of the Han-Li conjecture in Case (i), i.e. $ \eta_{1}' < 0 $.
\begin{theorem}\label{HL:thm2}
Let $ (\bar{M}, g) $ be a compact Riemannian manifold with non-empty smooth boundary $ \partial M $, $ n = \dim \bar{M} \geqslant 3 $. If $ \eta_{1}' < 0 $, there exist some negative constants $ \lambda, \zeta < 0 $ such that the PDE (\ref{HL:eqn1}) with $ S = \lambda $, $ H = \zeta $ admits a positive, smooth solution $ u $ on $ \bar{M} $. Equivalently, there exists a Yamabe metric $ \tilde{g} = u^{p-2} g $ such that $ R_{\tilde{g}} = \lambda < 0 $ and $ h_{\tilde{g}} = \zeta < 0 $.
\end{theorem}
\begin{proof} We construct lower solution and upper solutions of the following PDE
\begin{equation}\label{HL:eqn4}
-a\Delta_{g} u + R_{g} u = \lambda u^{p-1} \; {\rm in} \; M, \frac{\partial u}{\partial \nu} + \frac{2}{p-2} h_{g} u = \frac{2}{p-2} \zeta u^{\frac{p}{2}} \; {\rm on} \; \partial M
\end{equation}
for some appropriate choices of $ \lambda $ and $ \zeta $. According to the proof of the Han-Li conjecture in Theorem \ref{HL:thm1}, we may assume that $ h_{g} = h > 0 $ for some positive constant $ h $, since otherwise we apply pointwise conformal change to $ g $. Choosing $ \beta > 0 $ small enough so that the conclusion and the equation (\ref{HL:eqn2}) in Proposition \ref{HL:prop1} holds. Fix this $ \beta $. It follows that we can choose some constant $ \lambda < 0 $ such that
\begin{equation*}
-a\Delta_{g} \varphi + R_{g} \varphi = \eta_{1, \beta}' \varphi \leqslant \lambda \varphi^{p-1} \; {\rm in} \; M.
\end{equation*}
With fixed $ \lambda < 0 $, we choose $ \zeta < 0 $ with small enough absolute value so that the requirement in (\ref{pre:eqn17}) holds, we can also make $ \lvert \zeta \rvert $ even smaller if necessary so that
\begin{equation*}
\frac{\partial \varphi}{\partial \nu} + \frac{2}{p-2} h_{g} \varphi = -\frac{2}{p-2} \beta \varphi \leqslant \frac{2}{p-2} \zeta \varphi^{\frac{p}{2}}.
\end{equation*}
The inequality above holds for small enough $ \lvert \zeta \rvert $ since $ \zeta < 0 $. We then define
\begin{equation}\label{HL:eqn5}
u_{-} : = \varphi \; {\rm on} \; \bar{M}
\end{equation}
as a lower solution of (\ref{HL:eqn4}). For upper solution, we take
\begin{equation}\label{HL:eqn6}
u_{+} : = C \gg1 \; {\rm on} \; \bar{M}
\end{equation}
When the constant $ C $ is large enough, it is clear that
\begin{equation*}
-a\Delta_{g} u_{+} + R_{g} u_{+} = R_{g} C \geqslant \lambda C^{p-1} \; {\rm in} \; M
\end{equation*}
since both $ R_{g} $ and $ \lambda $ are negative. On $ \partial M $, we check that
\begin{equation*}
-\frac{\partial u_{+}}{\partial \nu} + \frac{2}{p-2} h_{g} u_{+} = \frac{2}{p-2} h C \geqslant 0 \geqslant \frac{2}{p-2} \zeta u_{+}^{\frac{p}{2}}.
\end{equation*}
We also require the constant $ C \geqslant \sup_{\bar{M}} \varphi = \sup_{\bar{M}} u_{-} $. Since both $ u_{+} $ and $ u_{-} $ are smooth, $ 0 < u_{-} \leqslant u_{+} $, we apply Theorem \ref{pre:thm4} and conclude that (\ref{HL:eqn4}) has a positive, smooth solution.
\end{proof}
\medskip

As we know, the negative first eigenvalue case is relatively easy case in the sense that the global lower solution and upper solution are not hard to find out. As we have shown in \cite{XU4}, \cite{XU5}, \cite{XU6}, \cite{XU7} and \cite{XU3}, it is not easy to apply monotone iteration scheme when the first eigenvalue is positive. In particular, there is a difference between prescribing constant and non-constant scalar curvature functions. In order to show the generalization of the Han-Li conjecture, we need the solution of a perturbed local Yamabe equation.
\begin{proposition}\label{HL:prop3}\cite[Prop.~3.3]{XU3}
Let $ (\Omega, g) $ be Riemannian domain in $\R^n$, $ n \geqslant 3 $, with $C^{\infty} $ boundary, and with ${\rm Vol}_g(\Omega)$ and the Euclidean diameter of $\Omega$ sufficiently small. Let $ \beta < 0 $ be any constant. Assume $ S_{g} < 0 $ within the small enough closed domain $ \bar{\Omega} $. Then for any $ \lambda > 0 $, the following Dirichlet problem
\begin{equation}\label{HL:eqn8}
-a\Delta_{g} u + \left(S_{g} + \beta \right) u = \lambda u^{p-1} \; {\rm in} \; \Omega, u = 0 \; {\rm on} \; \partial \Omega.
\end{equation}
has a real, positive, smooth solution $ u \in \calC^{\infty}(\Omega) \cap H_{0}^{1}(\Omega, g) $ in a very small domain $ \Omega $ that vanishes at $ \partial \Omega $.
\end{proposition}
We are ready to show the extension of the Han-Li conjecture in Case (ii), i.e. $ \eta_{1}' > 0 $. Note that there is a significant difference between constant and non-constant prescribing scalar curvatures. When $ S $ is not a globally constant, we can use the local solution in Proposition \ref{local:prop2} and the gluing procedure in \cite[Lemma~3.2]{XU6} to construct the lower and upper solutions. When $ S  = \lambda > 0 $ on $ M $, we have to introduce a perturbed Yamabe operator $ -a\Delta_{g} + R_{g} + \beta $ for $ \beta < 0 $, solving the perturbed Yamabe equation first, then take the limit as $ \beta \rightarrow 0^{-} $. In this case, the local solution of (\ref{HL:eqn8}) is the key.
\begin{theorem}\label{HL:thm3}
Let $ (\bar{M}, g) $ be a compact Riemannian manifold with non-empty smooth boundary $ \partial M $, $ n = \dim \bar{M} \geqslant 3 $. If $ \eta_{1}' > 0 $, there exist some negative constants $ \lambda> 0, \zeta < 0 $ such that the PDE (\ref{HL:eqn1}) with $ S = \lambda $, $ H = \zeta $ admits a positive, smooth solution $ u $ on $ \bar{M} $. Equivalently, there exists a Yamabe metric $ \tilde{g} = u^{p-2} g $ such that $ R_{\tilde{g}} = \lambda > 0 $ and $ h_{\tilde{g}} = \zeta < 0 $.
\end{theorem}
\begin{proof}
Here we have no restriction on the choice of initial mean curvature $ h_{g} $ for the construction of the upper solution since the target mean curvature is negative. We also have no restriction for the construction of the lower solution since the lower solution will be identically zero within a collar region containing $ \partial M $. By Proposition \ref{zero:prop2}, also see \cite[Thm.~4.6]{XU3} and \cite[Thm.~5.7]{XU4}, we may assume that the initial metric has scalar curvature $ R_{g} $ that is negative somewhere and mean curvature $ h_{g} > 0 $ everywhere on $ \partial M $. Note that the small enough region on which $ R_{g} < 0 $ can be arbitrarily chosen. The following argument is essentially the same as in \cite[Thm.~6.3, Thm.~6.4, Prop.~6.1]{XU5}, so we only give a very concise sketch here. By Proposition \ref{HL:prop2} above, we pick up a small enough constant $ \beta < 0 $ and consider the eigenvalue problem
\begin{equation*}
-a\Delta_{g} \varphi + \left(R_{g} + \beta \right) \varphi = \eta_{1, \beta}' \varphi \; {\rm in} \; M, \frac{\partial \varphi}{\partial \nu} + \frac{2}{p-2} h_{g} \varphi = 0 \; {\rm on} \; \partial M.
\end{equation*}
Choose $ \lambda > 0 $ so that
\begin{equation*}
\eta_{1, \beta}' \inf_{\bar{M}} \varphi > \lambda \cdot 2^{p-2} \cdot \sup_{\bar{M}} \varphi^{p-1}.
\end{equation*}
Fix this $ \lambda $. Note that we need the strict inequality above to allow some room for the gluing procedure in the construction of the upper solution. This can be done since we have assumed that $ R_{g} $ is negative somewhere; in addition, the introduction of $ \beta $ breaks down the conformal invariance.

We now apply Proposition \ref{HL:prop3} to obtain a local solution $ u_{0} $ of (\ref{HL:eqn8}) on a small enough domain $ \Omega $ on which $ R_{g} < 0 $, with the fixed $ f = \lambda $.
Set
\begin{equation}\label{HL:eqn11}
u_{-} : = \begin{cases} u_{0} & \; {\rm in} \; \Omega \\ 0 \; {\rm in} \; \bar{M} \backslash \bar{\Omega} \end{cases}.
\end{equation}
It is clear that $ u_{-} $ is a lower solution of
\begin{equation}\label{HL:eqn9}
-a\Delta_{g} u + \left( R_{g} + \beta \right) u = \lambda u^{p-1} \; {\rm in} \; M, \frac{\partial u}{\partial \nu} + \frac{2}{p - 2} h_{g} u = \frac{2}{p-2} \cdot \zeta u^{\frac{p}{2}} \; {\rm on} \; \partial M.
\end{equation}
Note that the boundary condition holds since $ u_{-} \equiv 0 $ within a collar region containing $ \partial M $. Furthermore, $ u_{-} \in H^{1}(M, g) \cap \calC^{0}(\bar{M}) $. For details, see e.g. Theorem 6.3 of \cite{XU5}. For upper solution, we glue the two functions $ u_{0} $ and $ \varphi $ in $ \Omega $ together to obtain a new function $ u_{1} \in \calC^{\infty}(\bar{\Omega}) $ which satisfies
\begin{equation*}
-a\Delta_{g} u_{1} + \left(R_{g} + \beta \right) u_{1} \geqslant \lambda u_{1}^{p-1} \; {\rm in} \; \Omega, u_{1} = \varphi \; {\rm on} \; \partial \Omega, u_{1} \geqslant u_{-} \; {\rm in} \; \bar{\Omega}.
\end{equation*}
The gluing strategy has been repeatedly used in many previous papers, for details, see e.g. Theorem 6.3 of \cite{XU5}. Set
\begin{equation}\label{HL:eqn12}
u_{+} : = \begin{cases} u_{1} & \; {\rm in} \; \Omega \\ \varphi \; {\rm in} \; \bar{M} \backslash \bar{\Omega} \end{cases}.
\end{equation}
Since $ u_{1} = \varphi $ near $ \partial \Omega $, we conclude that $ u_{+} $ is smooth on $ \bar{M} $. Due to the choice of $ \lambda $, it is clear that
\begin{equation*}
-a\Delta_{g} u_{+} + \left( R_{g} + \beta \right) u_{+} \geqslant \lambda u_{+}^{p-1} \; {\rm in} \; M.
\end{equation*}
Since $ \zeta < 0 $, the boundary condition satisfies
\begin{equation*}
\frac{\partial u_{+}}{\partial \nu} + \frac{2}{p-2} h_{g} u_{+} = 0 \geqslant \frac{2}{p-2} \zeta u_{+}^{\frac{p}{2}}, \forall \zeta < 0.
\end{equation*}
Furthermore, we have $ 0 \leqslant u_{-} \leqslant u_{+} $ with $ u_{-} \not\equiv 0 $. Choosing $ \zeta $ with small enough absolute value so that (\ref{pre:eqn17}) holds, we apply Theorem \ref{pre:cor1} and conclude that (\ref{HL:eqn9}) has a positive solution $ u_{\beta} \in \calC^{\infty}(\bar{M}) $ for appropriate choices of $ \lambda $ and $ \zeta $. Choosing a threshold $ \beta_{0} < 0 $, $ \lvert \beta_{0} \rvert $ small enough, every $ \beta \in (\beta_{0}, 0) $ associates with a solution $ u_{\beta} $ of (\ref{HL:eqn9}). By the same argument as in \cite[Prop.~6.1]{XU5}, we can show that 
\begin{equation*}
\lVert u_{\beta} \rVert_{\calL^{r}(M, g)} \leqslant C, \lVert u_{\beta} \rVert_{\calL^{p}(M, g)} \in [a, b], \forall \beta \in (\beta_{0}, 0).
\end{equation*}
Here $ r > p $, $ C, a, b $ are some positive constants. According to Arzela-Ascoli, there exists a subsequence of $ \lbrace u_{\beta} \rbrace $, uniformly bounded in $ \calC^{2, \alpha} $-norm for some $ \alpha \in (0, 1) $ due to standard bootstrapping method, that converges to a limit $ u $ in the classical sense. By the same argument in \cite[Thm.~6.4]{XU5}, the limit $ u \in \calC^{\infty}(\bar{M}) $ solves (\ref{HL:eqn1}) with $ S = \lambda > 0 $ and $ H = \zeta < 0 $, as chosen above.
\end{proof}
\medskip

\bibliographystyle{plain}
\bibliography{YamabessKWNS}

\begin{thebibliography}{10}

\bibitem{Niren4}
S.~Agmon, A.~Douglis, and L.~Nirenberg.
\newblock Estimates near the boundary for solutions of elliptic partial
  differential equstions satisfying general boundary conditions {I}.
\newblock {\em Commun. Pure Appl. Math}, 12:623--727, 1959.

\bibitem{Aubin}
T.~Aubin.
\newblock {\em Nonlinear Analysis on Manifolds. {M}onge-{A}mp\'ere
  {E}quations.}
\newblock Grundlehren der mathematischen Wissenschaften. Springer, Berlin,
  Heidelberg, New York, 1982.

\bibitem{BC}
A.~Bahri and J.~Coron.
\newblock On a nonlinear elliptic equation involving the critical {S}obolev
  exponent: The effect of the topology of the domain.
\newblock {\em Commun. Pure Appl. Math}, 41(3):253--294, 1988.

\bibitem{BE}
J.~Bourguignon and J.~Ezin.
\newblock Scalar curvature functions in a conformal class of metrics and
  conformal transformations.
\newblock {\em Trans. Am. Math. Soc.}, 301(2):723--736, 1987.

\bibitem{BM}
S.~Brendle and F.~Marques.
\newblock Recent progress on the {Y}amabe problem.
\newblock {\em arXiv:1040.4960}.

\bibitem{AC}
A.~Chang.
\newblock Nirenberg's problem.
\newblock {\em Not. Am. Math. Soc.}, 63(2):128--130, 2016.

\bibitem{Che}
P.~Cherrier.
\newblock Probl\'emes de {N}eumann non lin\'eaires sur les vari\'et\'es
  {R}iemanniannes.
\newblock {\em J. Funct. Anal.}, 57:154--206, 1984.

\bibitem{CFP}
M.~Clapp, J.~Faya, and A.~Pistoia.
\newblock Positive solutions to a supercritical elliptic problem which
  concentrate along a thin spherical hole.
\newblock {\em J. Anal. Math.}, 126:341--357, January 2015.

\bibitem{Escobar2}
J.~Escobar.
\newblock Conformal deformation of a {R}iemannian metric to a scalar flat
  metric with constant mean curvature on the boundary.
\newblock {\em Ann. of Math. (2)}, 136(1):1--50, 1992.

\bibitem{ESC}
J.~Escobar.
\newblock The {Y}amabe problem on manifolds with boundary.
\newblock {\em J. Differential Geom.}, 35:21--84, 1992.

\bibitem{ESC2}
J.~Escobar.
\newblock Conformal deformation of a {R}iemannian metric to a constant scalar
  curvature metric with constant mean curvature on the boundary.
\newblock {\em Indiana Univ. Math. J.}, 45(4):917--943, 1996.

\bibitem{ESS}
J.~Escobar and R.~Schoen.
\newblock Conformal metrics with prescribed scalar curvature.
\newblock {\em Invent. Math.}, 86(2):243--254, 1986.

\bibitem{HL}
Z.~Han and Y.~Li.
\newblock The existence of conformal metrics with constant scalar curvature and
  constant boundary mean curvature.
\newblock {\em Comm. Anal. Geom.}, 8(4):809--869, 2000.

\bibitem{KW2}
J.~Kazdan and F.~Warner.
\newblock Curvature functions for compact 2$-$manifolds.
\newblock {\em Ann. of Math.}, 99:14--47, 1974.

\bibitem{KW3}
J.~Kazdan and F.~Warner.
\newblock Existence and conformal deformations of metrices with prescribed
  {G}aussian and scalar curvatures.
\newblock {\em Ann. of Math. (2)}, 101(2):317--331, 1975.

\bibitem{KW}
J.~Kazdan and F.~Warner.
\newblock Scalar curvature and conformal deformation of {R}iemannian structure.
\newblock {\em J. Differential Geom.}, 10:113--134, 1975.

\bibitem{PL}
J.~Lee and T.~Parker.
\newblock The {Y}amabe problem.
\newblock {\em Bull. Amer. Math. Soc. (N.S.)}, 17(1):37--91, 1987.

\bibitem{YYL}
Y.~Li.
\newblock Prescribing scalar curvature on $ \mathbb{S}^{n} $ and related
  problems, part {II}: Existence and compactness.
\newblock {\em Commun. Pure Appl. Math}, 49(6):541--597, 1996.

\bibitem{MaMa}
A.~Malchiodi and M.~Mayer.
\newblock Prescribing {M}orse scalar curvatures: Pinching and {M}orse theory.
\newblock {\em Commun. Pure Appl. Math}, 2021.

\bibitem{XU2}
S.~Rosenberg and J.~Xu.
\newblock Solving the {Y}amabe problem by an iterative method on a small
  {R}iemannian domain.
\newblock {\em ar{X}iv:2110.14543}.

\bibitem{SY}
R.~Schoen and S.-T. Yau.
\newblock Conformally flat manifolds, {K}leinian groups and scalar curvature.
\newblock {\em Invent. Math.}, 92:47--71, 1988.

\bibitem{T}
M.~Taylor.
\newblock {\em Partial Differential Equations {I}}.
\newblock Springer-Verlag, New York, New York, 2011.

\bibitem{T3}
M.~Taylor.
\newblock {\em Partial Differential Equations {III}}.
\newblock Springer-Verlag, New York, New York, 2011.

\bibitem{WANG}
X.~Wang.
\newblock Existensce of posotive solutions to nonlinear equations involving
  critical sobolev exponents.
\newblock {\em Acta Mathematica Sinica, New Series}, 8(3):273--291, 1992.

\bibitem{XU4}
J.~Xu.
\newblock The boundary {Y}amabe problem, {I}: Minimal boundaray case.
\newblock {\em ar{X}iv:2111:03219}.

\bibitem{XU5}
J.~Xu.
\newblock The boundary {Y}amabe problem, {II}: General constant mean curvature
  case.
\newblock {\em ar{X}iv:2112.05674}.

\bibitem{XU}
J.~Xu.
\newblock Iterative methods for globally {L}ipschitz nonlinear {L}aplace
  equations, ar{X}iv:1911.10192.
\newblock {\em Submitted}.

\bibitem{XU6}
J.~Xu.
\newblock Prescribed scalar curvature on compact manifolds under conformal
  deformation.
\newblock {\em ar{X}iv:2205.15453}.

\bibitem{XU7}
J.~Xu.
\newblock Prescribed scalar curvature problem under conformal deformation of a
  {R}iemannian metric with {D}irichlet boundary condition.
\newblock {\em ar{X}iv:2208.11318}.

\bibitem{XU3}
J.~Xu.
\newblock Solving the {Y}amabe-type equations on closed manifolds by iteration
  schemes.
\newblock {\em ar{X}iv: 2110.15436}.

\end{thebibliography}

\end{document}